\documentclass[11pt]{article}
\usepackage{xcolor,amsmath,mathrsfs,amssymb,accents}
\usepackage{amsmath,amsfonts,amsthm,epsfig,graphicx,epstopdf}
\usepackage{pstricks,pst-plot,multido}

\numberwithin{equation}{section}

\newtheorem{theorem}{Theorem}[section]
\newtheorem{lemma}[theorem]{Lemma}
\newtheorem{definition}[theorem]{Definition}
\newtheorem{remark}[theorem]{Remark}
\newtheorem{proposition}[theorem]{Proposition}
\newtheorem{corollary}[theorem]{Corollary}

\newenvironment{Proof}{\removelastskip\par\medskip 
\noindent{\em Proof.}
\rm}{\penalty-20\null\hfill$\square$\par\medbreak}

\newcommand{\Q}{\mathbb{Q}}
\newcommand{\N}{\mathbb{N}}
\newcommand{\R}{\mathbb{R}}

\newcommand{\Leb}[1]{{\mathscr L}^{#1}} 

\renewcommand{\O}{\Omega}

\newcommand{\BorelSets}[1]{\mathcal B(#1)}
\newcommand{\Probabilities}[1]{\mathscr P\bigl(#1\bigr)} 
\newcommand{\Measuresp}[1]{\mathscr M_+\bigl(#1\bigr)} 

\renewcommand\div{\operatorname{div}}

\newcommand{\supp}{\operatorname{supp}}

\newcommand{\hit}[2]{{\sf h}_{#1}({#2})}

\DeclareMathOperator*{\esssup}{ess\,sup}

\newcommand{\res}{\mathop{\hbox{\vrule height 7pt width .5pt depth 0pt
\vrule height .5pt width 6pt depth 0pt}}\nolimits}

 \newcommand{\bb}{{\mbox{\boldmath$b$}}}
 \newcommand{\cc}{{\mbox{\boldmath$c$}}}
 
 \newcommand{\ee}{{\mbox{\boldmath$e$}}}

 \newcommand{\tauV}{{\kern-3pt\tau}}
 \newcommand{\sbb}{{\mbox{\scriptsize\boldmath$b$}}}

 \newcommand{\sXX}{{\mbox{\scriptsize\boldmath$X$}}}
 \newcommand{\sYY}{{\mbox{\scriptsize\boldmath$Y$}}}
 \newcommand{\ssXX}{{\scriptscriptstyle \mbox{\scriptsize\boldmath$X$}}}
   \newcommand{\sox}{\Omega,{\mbox{\scriptsize\boldmath$X$}}}

 \newcommand{\XX}{{\mbox{\boldmath$X$}}}
  \newcommand{\YY}{{\mbox{\boldmath$Y$}}}
 
 \newcommand{\oVVVk}{\overline{\mbox{\boldmath$V$}}\kern-3pt}
 \newcommand{\tVVVk}{\tilde{\mbox{\boldmath$V$}}\kern-3pt}
\newcommand{\WW}{{\mbox{\boldmath$W$}}}

 \newcommand{\llambda}{{\mbox{\boldmath$\lambda$}}}

 \newcommand{\seeta}{{\mbox{\scriptsize\boldmath$\eta$}}}

 \newcommand{\eeta}{{\mbox{\boldmath$\eta$}}}
 \newcommand{\bbb}{b}
 \newcommand{\eps}{\varepsilon}

\newcommand{\BB}{{\mbox{\boldmath$B$}}}

\begin{document}

\title{Existence and uniqueness of maximal regular\\
flows for non-smooth vector fields
}
\date{}

\author{Luigi Ambrosio\
   \thanks{Scuola Normale Superiore, Pisa. email: \textsf{l.ambrosio@sns.it}}
   \and
   Maria Colombo\
   \thanks{Scuola Normale Superiore, Pisa. email: \textsf{maria.colombo@sns.it}}
 \and
   Alessio Figalli\
   \thanks{University of Texas at Austin. email:
   \textsf{figalli@math.utexas.edu}}
   }

\maketitle

\begin{abstract}
In this paper we provide a complete analogy between the Cauchy-Lipschitz and the DiPerna-Lions theories for ODE's,
by developing a local version of the DiPerna-Lions theory. More precisely,
we prove existence and uniqueness of a maximal regular
flow for the DiPerna-Lions theory using only local regularity and summability assumptions on the vector field, in analogy with the classical theory, which
uses only local regularity assumptions. We also study the behaviour of the ODE trajectories before the maximal existence time. Unlike the Cauchy-Lipschitz
theory, this behaviour crucially depends on the nature of the bounds imposed on the spatial divergence of the vector field. In particular, a global assumption
on the divergence is needed to obtain a proper blow-up of the trajectories.
\end{abstract} 

\tableofcontents

\section{Introduction}

Given a vector field $\bb(t,x)$ in $\R^d$,
the theory of DiPerna-Lions, introduced in the seminal paper \cite{lions}, provides existence and uniqueness of the
flow (in the almost everywhere sense, with respect to Lebesgue measure $\Leb{d}$) under weak regularity assumptions on $\bb$,
for instance when $\bb(t,\cdot)$ is Sobolev \cite{lions} or $BV$
\cite{ambrosio} and satisfies global bounds on the divergence. In this respect, this theory could be considered
as a weak Cauchy-Lipschitz theory for ODE's. This analogy is confirmed by many global existence results, by a kind
of Lusin type approximation of DiPerna-Lions flows by Lipschitz flows \cite{amcriman,crippade}, and even by differentiability
properties of the flow \cite{Lebris}. However, this analogy is presently not perfect, and the main aim of this paper is to fill this gap.

Indeed, the Cauchy-Lipschitz theory is not only pointwise but also purely local, meaning that existence and uniqueness
for small intervals of time depend {\it only} on local regularity properties of the vector fields $\bb(t,x)$. On the other hand,
not only the DiPerna-Lions theory is an almost everywhere theory (and this really seems to be unavoidable) but also the
existence results for the flow depend on {\it global} in space growth estimates on $|\bb|$, the most typical one being
\begin{equation}\label{eqn:globalDP}
\frac{|\bb(t,x)|}{1+|x|}\in L^1\bigl((0,T);L^1(\R^d)\bigr)+L^1\bigl((0,T);L^\infty(\R^d)\bigr).
\end{equation}
This is in contrast with the fact that the so-called ``renormalization property'', which plays a key role in the theory,
seems to depend only on local properties of $\bb$, because it deals with distributional solutions to a continuity/transport
equation with a source term: as a matter of fact, it is proved using only local regularity properties of $\bb$.

Given an open set $\O\subset\R^d$, in this paper we consider vector fields $\bb:(0,T)\times\O\to\R^d$ satisfying only the 
local integrability property $\int_0^T\int_{\O'}|\bb|dxdt<\infty$ for all $\O'\Subset\O$, a local one-sided bound on the distributional
divergence, and the property
that the continuity equation with velocity $\bb$ is well-posed in the class of nonnegative bounded and
compactly supported functions in $\O$. As illustrated in Remark~\ref{rmk:manycases}, this last assumption is fulfilled in many cases
of interest and it is known to be deeply linked to the uniqueness of the flow; in addition, building on the superposition principle 
(Theorem~\ref{thm:superpo}), it is proved in the appendix that even this assumption is purely  local, as well as the other two ones
concerning integrability and bounds on divergence. 

Under these three assumptions we prove existence
of a unique {\it maximal regular} flow $\XX(t,x)$ in $\O$, defined up to a maximal time $T_{\sox}(x)$ which is positive $\Leb{d}$-a.e.
in $\O$, with 
\begin{equation}\label{eqn:blowupintro}
\limsup_{t\uparrow T_{\sox}(x)}V_\O(\XX(t,x))=\infty\qquad\text{for $\Leb{d}$-a.e. $x\in \{T_{\sox}<T\}$.}
\end{equation}
Here $V_\O:\O\to [0,\infty)$ is a given continuous ``confining potential'', namely with $V(x)\to \infty$ as $x\to\partial\O$; hence, 
\eqref{eqn:blowupintro} is a synthetic way to state that, for any $\O'\Subset\O$, $\XX(t,x)$ does not intersect $\O'$ for $t$
close to $T_{\sox}(x)$.  

In our axiomatization, which parallels the one of \cite{ambrosio} and slightly differs from the one of the
DiPerna-Lions theory (being only based on one-sided bounds on divergence and independent of the
semigroup property), ``maximal'' refers to \eqref{eqn:blowupintro}, while ``regular'' means the existence of constants $C(\O',\XX)$ such that
\begin{equation}\label{eqn:regularmeans}
\int_{\O'\cap\{{\sf h}_{\O'}>t\}}\phi(\XX(t,x))\,dx\leq C(\O',\XX)\int_{\R^d}\phi(y)\,dy
\qquad\text{for all $\phi\in C_c(\R^d)$ nonnegative}
\end{equation}
for all $t\in [0,T]$, $\O'\Subset\O$, where ${\sf h}_{\O'}(x)\in [0,T_{\sox}(x)]$ is the first time that $\XX(\cdot,x)$ hits $\R^d \setminus \O'$.
Under global bounds on the divergence, \eqref{eqn:regularmeans} can be improved to
\begin{equation}\label{eqn:regularmeans2}
\int_{\O\cap\{T_{\sox}>t\}}\phi(\XX(t,x))\,dx\leq C_*\int_{\R^d}\phi(y)\,dy
\qquad\text{for all $\phi\in C_c(\R^d)$ nonnegative}
\end{equation}
for all $t\in [0,T]$,
but many structural properties can be proved with \eqref{eqn:regularmeans} only.

Uniqueness of the maximal regular flow 
follows basically from the ``probabilistic'' techniques developed in \cite{ambrosio}, which allow one to transfer uniqueness
results at the level of the PDE (the continuity equation), here axiomatized, into uniqueness results at the level of the ODE. 
Existence follows by analogous techniques; the main new difficulty here is that even if we truncate $\bb$ with a
$C^\infty_c(\O)$ cut-off function, the resulting vector field has not divergence in $L^\infty$ (just $L^1$, actually, when
$|\bb_t|\notin L^\infty_{\rm loc}(\O)$), 
hence the standard theory is not applicable. Hence, several new ideas and techniques need to be introduced to handle this new situation. These results are achieved
in Section~\ref{sec:local}.

Besides existence and uniqueness, we prove a natural semigroup property for $\XX$ and for $T_{\sox}$ and some
additional properties which depend on \emph{global} bounds on the divergence, more precisely on \eqref{eqn:regularmeans2}. 
The first property, well known in the classical setting, is properness of the blow-up, namely this enforcement of \eqref{eqn:blowupintro}:
\begin{equation}\label{eqn:blowupintrobis}
\lim_{t\uparrow T_{\sox}(x)}V_\O(\XX(t,x))=\infty\qquad\text{for $\Leb{d}$-a.e. $x\in \{T_{\sox}<T\}$.}
\end{equation}
In other terms, for any $\O'\Subset\O$ we have that $\XX(t,x)\notin\O'$ for $t$ sufficiently close to $T_{\sox}(x)$.

In $\O=\R^d$, $d\geq 2$, we also provide an example of an autonomous Sobolev vector field
showing that \eqref{eqn:blowupintro} cannot be improved to \eqref{eqn:blowupintrobis} when only local bounds on divergence
are present. We also discuss the 2-dimensional case for $BV_{\rm loc}$ vector fields.
The second property is the continuity of $\XX(\cdot,x)$ up to $T_{\sox}(x)$, discussed in
Theorem~\ref{thm:pasquetta}, and sufficient conditions for $T_{\sox}(x)=T$.  

%

Finally, we discuss the stability properties
of $\XX$ before the blow-up time $T_{\sXX}$
with respect to perturbations of $\bb$. The results of this paper will be applied to describe the lagrangian structure of weak solutions of the Vlasov-Poisson equation and to prove existence of weak solutions with $L^1$ summability of the initial datum.

%
%

\smallskip
\noindent {\bf Acknowledgement.} The first and third author acknowledge the support
of the ERC ADG GeMeThNES, the second author has been partially supported by PRIN10 
grant from MIUR for the project Calculus of Variations, the third author has been partially supported 
by NSF Grant DMS-1262411.
This material is also based upon work supported by the National Science
Foundation under Grant No. 0932078 000, while the second and third authors were in residence 
at the Mathematical Sciences Research Institute in Berkeley, California, during the fall semester of 2013.

\section{Notation and preliminary results}

We mostly use standard notation, denoting by $\Leb{d}$ the Lebesgue measure in $\R^d$, and by $f_\#\mu$ the push-forward
of a Borel nonnegative measure $\mu$ under the action of a Borel map $f$, namely $f_\#\mu(B)=\mu(f^{-1}(B))$ for any
Borel set $B$ in the target space. We denote by $\BorelSets{\R^d}$ the family of all Borel sets in $\R^d$. 
In the family of positive finite measures in an open set $\Omega$, we will consider both the weak topology induced by
the duality with $C_b(\Omega)$ that we will call {\it narrow} topology, and the {\it weak} topology induced by $C_c(\Omega)$.

If $J\subset\R$ is an interval and $t\in J$, we denote by $e_t:C(J;\R^d)\to\R^d$ the evaluation map at time $t$, namely
$e_t(\eta):=\eta(t)$ for any continuous curve $\eta:J \to \R^d$.
The rest of the section is devoted to the discussion of preliminary results on solutions to the continuity equation, with statements
and proofs adapted to our problem.
Also, $\Measuresp{\R^d}$ will denote the space of finite Borel measures on $\R^d$, while 
$\Probabilities{\R^d}$ denotes the space of probability measures.

Let us fix $T\in (0,\infty)$ and 
consider a weakly continuous family $\mu_t\in\Measuresp{\R^d}$, $t\in [0,T]$, solving in the sense of distributions the continuity equation
$$
\frac{d}{dt}\mu_t+\nabla\cdot (\bb_t\mu_t)=0\qquad\text{in $(0,T)\times\R^d$}
$$
for a Borel vector field $\bb:(0,T)\times\R^d\to\R^d$, locally integrable with respect to the
space-time measure $\mu_tdt$. When we restrict ourselves to probability measures $\mu_t$, then weak and narrow continuity
w.r.t. $t$ are equivalent; analogously, we may equivalently consider compactly supported test functions $\varphi(t,x)$ in the
weak formulation of the continuity equation, or functions with bounded $C^1$ norm whose support is contained in $I\times\R^d$ with $I\Subset (0,T)$.

We now recall the so-called superposition principle. We prove it under the general assumption that $\mu_t$ 
may a priori vanish for some $t\in [0,T]$, but satisfies \eqref{hp:super}; a posteriori we see that $\mu_t\in\Probabilities{\R^d}$ for every $t\in [0,T]$.
 The superposition principle will play a role in the proof of the comparison principle stated in 
 Proposition~\ref{prop:uni_continuity} and in the blow-up criterion of Theorem~\ref{thm:noblowup}.
 In connection with earlier analogous results, we will exploit the argument in the proof of \cite[Theorem~12]{bologna}, 
 stated for probability measures but that still works under the  assumption that $\mu_t\in\Measuresp{\R^d}$ 
 and $\mu_t(\R^d)>0$ for every $t\in [0,T]$. See also \cite[Theorem~8.2.1]{amgisa}, where a proof is presented in the even more 
special case of $L^p$ integrability on $\bb$  for some $p>1$
$$
\int_0^T\int_{\R^d}{|\bb(t,x)|^p}\,d\mu_t(x)\,dt<\infty.
$$ 

\begin{theorem}[Superposition principle and approximation]\label{thm:superpo}
Let $\bb:(0,T)\times\R^d\to \R^d$ be a Borel vector field. Let $\mu_t\in\Measuresp{\R^d}$, $0\leq t\leq T$, with $\mu_t$ weakly continuous in $[0,T]$ solution to the equation 
$\frac{d}{dt}\mu_t+{\rm div\,}(\bb\mu_t)=0$ in $(0,T) \times \R^d$, with
\begin{equation}
\label{hp:super}
\int_0^T\int_{\R^d}\frac{|\bb(t,x)|}{1+|x|}\,d\mu_t(x)\,dt<\infty.
\end{equation}
Then there exists $\eeta\in\Measuresp{C([0,T];\R^d)}$ satisfying:
\begin{itemize}
\item[(i)] $\eeta$ is concentrated on absolutely continuous curves
$\eta$ in $[0,T]$, solving the ODE $\dot\eta=\bb(t,\eta)$ $\Leb{1}$-a.e. in $(0,T)$;
\item[(ii)] $\mu_t=(e_t)_\#\eeta$ (so, in particular, $\mu_t(\R^d)= \mu_0(\R^d)$) for all $t \in [0,T]$.\end{itemize}
Moreover, there exists a family of measures $\mu_t^R\in\Measuresp{\R^d}$, narrowly continuous in $[0,T]$, solving the continuity equation and supported on
$\overline{B}_R$, such that $\mu^R_t\uparrow\mu_t$ as $R \to \infty$ for all $t\in [0,T]$.
\end{theorem}
\begin{Proof}  
If $\mu_t = 0$ for every $t\in [0,T]$, the statement is trivial. Otherwise, there exists $t_0 \in [0,T]$ such that $\mu_{t_0}(\R^d)>0$. Since \cite[Theorem~12]{bologna} proves the result when $\mu_t(\R^d) >0$ for every $t \in [0,T]$, we are left to verify this property. More precisely, we prove that $\mu_t(\R^d) >0$ for every $t \in [t_0,T]$; the same argument applies (arguing with the backward equation starting from $t_0$) on $[0,t_0]$.

Since the function $t \to \mu_t(\R^d)$ is lower semicontinuous, let us consider, by contradiction, the smallest $t_1\in (t_0,T]$ such that $\mu_{t_1}(\R^d)=0$. As it is easy to check,
we can apply the proof of \cite[Theorem~12]{bologna} on any time interval $[t_0, t_1-n^{-1}]$ to deduce that
$\mu_t=(e_t)_\#\eeta^n$ for all $t \in [t_0,t_1-n^{-1}]$, where $\eeta^n\in\Probabilities{C([t_0,t_1-n^{-1}];\R^d)}$ satisfies (i).
To find a contradiction, we want to show that the weak limit $\mu_{t_1}$ of $\mu_{t_1-n^{-1}}$ as $n\to \infty$ is nonzero.

Let $r>0$ be such that $\mu_{t_0}(B_r)>0$ and let us consider for $n>t_1^{-1}$
$$\eeta^{n,r} = \eeta^n \res \{\gamma \in C([t_0,t_1-n^{-1}];\R^d):\ \gamma(0) \in B_r\}.$$
We have that $\eeta^{n,r}(C([t_0,t_1-n^{-1}];\R^d)) = \mu_{t_0}(B_r)>0$ and since $\eeta^{n,r} \leq \eeta^n$, we deduce that 
\begin{equation}
\label{eqn:sotto-nonzero}
(e_t)_\#\eeta^{n,r} \leq \mu_t \qquad \forall \,t \in [t_0,t_1-n^{-1}].
\end{equation}
Since $|\gamma(0)| \leq r$ for $\eeta^{n,r}$-a.e. $\gamma$, 
$\eeta^{n,r}$ is concentrated on integral curves of $\bb$ and \eqref{eqn:sotto-nonzero} holds, we obtain that
\begin{equation*}
\begin{split}
\int_{\R^d}\log\Big(\frac{1+|x|}{1+r}\Big)\,d[(e_{t_1-n^{-1}})_\#\eeta^{n,r}] (x) 
&\leq
\int_{\R^d} \log\Big(\frac{1+|\gamma(t_1-n^{-1})|}{1+|\gamma(0)|} \Big) \, d\eeta^{n,r}(\gamma)
\\
&\leq 
\int_{\R^d} \int_{t_0}^{t_1-n^{-1}}\frac{d}{dt} \log({1+|\gamma(t)|}) \, dt \, d\eeta^{n,r}(\gamma)
\\
&\leq 
\int_{\R^d} \int_{t_0}^{t_1-n^{-1}}\frac{|\dot\gamma(t)|}{1+|\gamma(t)|} \, dt \, d\eeta^{n,r}(\gamma)
\\
&=
\int_{t_0}^{t_1-n^{-1}}\int_{\R^d}\frac{|\bb(t,\gamma(t))|}{1+|\gamma(t)|}\,dt \, d\eeta^{n,r}(\gamma)
\\
&=
\int_{t_0}^{t_1-n^{-1}}\int_{\R^d}\frac{|\bb(t,x)|}{1+|x|} \,d[(e_t)_\#\eeta^{n,r}] (x) \,dt
\\
&\leq
\int_{0}^T\int_{\R^d}\frac{|\bb(t,x)|}{1+|x|}\,d\mu_t(x)\,dt.
\end{split}
\end{equation*}
Hence
\begin{equation}
\label{eqn:coercive-nonzero}\sup_{n\in \N} \int_{\R^d}\log({1+|x|})\,d[(e_{t_1-n^{-1}})_\#\eeta^{n,r}](x)<\infty.
\end{equation}
Since the measures $(e_{t_1-n^{-1}})_\#\eeta^{n,r}$ have mass $\mu_{t_0}(B_r)$, we deduce from the tightness estimate
\eqref{eqn:coercive-nonzero} that, up to a subsequence, they converge to a narrow limit $\nu$ such that $\nu(\R^d) = \mu_{t_0}(B_r)>0$.
Hence, by \eqref{eqn:sotto-nonzero} applied with $t={t_1-n^{-1}}$ 
the weak limit $\mu_{t_1}$ of $\mu_{t_1-n^{-1}}$ as $n\to \infty$ is nonzero, a contradiction.

This proves that $\mu_t(\R^d)\geq c>0$ for all $t \in [{t_0},T]$.

The last statement
can simply be obtained by restricting $\eeta$ to the class of curves contained in $\overline{B}_R$
for all $t \in [0,T]$ to obtain positive finite
measures $\eeta^R\leq\eeta$ which satisfy $\eeta^R\uparrow\eeta$, and then defining $\mu^R_t:=(e_t)_\#\eeta^R$.
\end{Proof}

\section{Integrability and uniqueness of bounded solutions of the continuity equation}\label{sec:ass-a-c}

Given a closed interval $I\subset\R$ and an open set $\O\subset\R^d$, 
let us define the class $\mathcal L_{I,\O}$ of all nonnegative functions which are essentially 
bounded, nonnegative, and compactly supported in $\O$:
\begin{equation}\label{eqn:mathcalL-o}
\mathcal L_{I,\O}:=L^\infty\bigl(I;L^\infty_+(\O)\bigr)\cap\bigl\{w:\
\text{$\supp w$ is a compact subset of $I\times \O$}\bigr\}.  
\end{equation}
We say that $\rho\in\mathcal L_{I,\O}$ is weakly$^*$ continuous if there is a representative
$\rho_t$ with $t\mapsto\rho_t$ continuous in $I$ w.r.t. the weak$^*$ topology of $L^\infty(\O)$.
Notice that, in the class $\mathcal L_{I,\O}$, weak$^*$ continuity of $\rho$ is equivalent to the 
narrow continuity of the corresponding measures $\mu_t:=\rho_t\Leb{d}\in\Measuresp{\R^d}$.

For $T\in (0,\infty)$ we are given a Borel vector field $\bb:(0,T)\times\O\to\R^d$ satisfying:
\begin{itemize}
\item[(a-$\O$)] $\int_0^T\int_{\O'}|\bb(t,x)|\,dxdt<\infty$ for any $\O'\Subset\O$;
\item[(\bbb-$\O$)] for any nonnegative $\bar\rho\in L^\infty_+(\O)$ with compact support in $\Omega$ and 
any closed interval $I=[a,b]\subset [0,T]$, the continuity equation
\begin{equation}
\label{eqn:ce}
\frac{d}{dt}\rho_t+{\rm div\,}(\bb\rho_t)=0\qquad\text{in $(a,b)\times\O$}
\end{equation}
has at most one weakly$^*$ continuous solution $I\ni t\mapsto\rho_t\in\mathcal L_{I,\O}$ with $\rho_a=\bar\rho$.
\end{itemize}
\begin{remark}\label{rmk:manycases} {\rm Assumption (\bbb-$\O$) is known to be true in many cases. The following list does not pretend
to be exhaustive:

-- Sobolev vector fields \cite{lions}, $BV$ vector fields whose divergence is a locally integrable function in space
\cite{bouchut,ambrosio,collerner1,collerner}, some classes of vector fields of bounded deformation \cite{amcriman}; 

-- vector fields $\BB(x,y)=(\bb_1(x,y),\bb_2(x,y))$
with different regularity w.r.t. $x$ and $y$ \cite{Lebris,lerner};

-- two-dimensional Hamiltonian vector fields
\cite{abc1} (within this class, property (\bbb-$\O$) has been characterized in terms of the so-called weak Sard property);

-- vector fields arising from the convolution of $L^1$ functions with singular integrals \cite{poly,boucrising}.
{In this case, the authors proved uniqueness of the regular lagrangian flow associated to $\bb$; we outline in the next remark how to obtain the eulerian uniqueness property (\bbb-$\O$) following their argument.}
}\end{remark}

\begin{remark} {\rm 
Under the assumptions on the vector field $\bb$ considered in \cite{boucrising}, the authors proved 
in \cite[Theorem 6.2]{boucrising} the uniqueness of the lagrangian flow. 
In their key estimate, the authors take two regular lagrangian flows $X$ and $Y$, provide an upper and lower bound for the quantity
\begin{equation}
\label{eqn:cdl-func}
\Phi_\delta (t) := \int \log \Big( 1+ \frac{|X(t,x)- Y(t,x)|}{\delta} \Big) \, dx \qquad t\in [0,T]
\end{equation}
in terms of a parameter $\delta>0$, and eventually let $\delta \to 0$.
To show that property (\bbb-$\O$) holds, we consider two nonnegative bounded solutions of the continuity equation with the same initial datum which are compactly supported in $[a,b] \times \O$. By Theorem \ref{thm:superpo} there exist
$\eeta^1,\eeta^2\in\Probabilities{C([a,b];\R^d)}$ which are concentrated on absolutely continuous solutions
$\eta \in AC([a,b]; \O)$ of the ODE $\dot\eta=\bb(t,\eta)$ $\Leb{1}$-a.e. in $(a,b)$, and satisfy $(e_t)_\#\eeta^i \leq C \Leb{d}$ for any $t\in [a,b]$, $i=1,2$. Moreover, we have that $(e_a)_\# \eeta^1 = (e_a)_\# \eeta^2$.
Given $\delta>0$, we consider the quantity
\begin{equation}
\label{eqn:cdl-func-new}
\Psi_\delta(t) := \int_{\O} \int \int \log \Big( 1+ \frac{| \gamma(t) - \eta(t)|} {\delta} \Big) d \eeta^{1}_{x}( \gamma) d\eeta^{2}_{x}(\eta) \, d [(e_0)_\# \eeta^1](x)  
 \qquad t\in [a,b],
\end{equation}
where $\eeta^1_{x}$, $\eeta^2_x$ are the disintegrations of $\eeta^1$ and $\eeta^2$ with respect to the map $e_a$.
Since $\eeta^1$ and $\eeta^2$ are concentrated on curves in $C([a,b];\O)$, to show that $\eeta^1= \eeta^2$ we can neglect the behavior of $\bb$ outside $\O$. Following the same computations of  \cite{boucrising} with the functional \eqref{eqn:cdl-func-new} instead of \eqref{eqn:cdl-func}, we show that $\eeta^{1}_{x} = \eeta^{2}_{x}$ for $(e_a)_\# \eeta^{1}$-a.e. $x\in \O$ and this implies the validity of property (\bbb-$\O$).
}\end{remark}

More recently, these well-posedness results have also been extended to 
infinite-dimensional vector fields (see \cite{amfigaflow} and the bibliography therein). It is interesting to observe 
that the uniqueness assumption in (\bbb-$\O$) actually implies the validity
of a comparison principle.

\begin{proposition}[Comparison principle]\label{prop:uni_continuity}
If (a-$\O$) and (\bbb-$\O$) are satisfied, then the following implication holds:
$$
\rho^1_0\leq\rho^2_0\qquad\Longrightarrow\qquad\rho^1_t\leq\rho^2_t\quad\forall \,t\in [0,T]
$$
for all weakly$^*$ continuous solutions of \eqref{eqn:ce} in the class $\mathcal L_{[0,T],\O}$.
\end{proposition}
\begin{Proof} Let
$\eeta^i$ be representing $\mu^i_t:=\rho^i_t\Leb{d}$ according to Theorem \ref{thm:superpo},
and let $\eeta_x^i$ be the conditional probability measures induced by $e_0$, that is
$$
\int F(\eta)\,d\eeta^i = \int_{\R^d} \biggl(\int F(\eta)\,d\eeta_x^i\biggr)\,d\mu_0^i(x)\qquad \forall\,F:C([0,T];\R^d)\to \R \text{ bounded},
$$
or (in a compact form)
$\eeta^i(d\eta) = \int \eeta_x^i(d\eta)\,d \mu_0^i(x).$
Defining
$$
\tilde\eeta(d\eta):=\int \eeta^2_x \, d\mu^1_0(x), \qquad \tilde\mu_t:=(e_t)_\#\tilde\eeta,$$
because
$\mu^1_0\leq\mu^2_0$, we get $\tilde \eeta \leq \eeta^1$. Moreover, the densities of measures $\tilde \mu_t$ and $\mu_t^1$ provide two elements in $\mathcal L_{[0,T],\O}$, solving the continuity equation 
with the same initial condition $\mu^1_0$.
Therefore assumption (\bbb-$\O$) gives $\tilde\mu_t=\mu^1_t$ for all $t\in [0,T]$, and $\mu^1_t = \tilde \mu_t =(e_t)_\#\tilde\eeta \leq (e_t)_\#\eeta^2=\mu^2_t$ for all $t\in [0,T]$, as desired.
\end{Proof}

\begin{theorem}\label{thm:nosplitting}
Assume that $\bb$ satisfies (a-$\O$) 
and (\bbb-$\O$), and let $\llambda\in\Probabilities{C([0,T];\R^d)}$ 
satisfy:
\begin{itemize}
\item[(i)] $\llambda$ is concentrated on
$$
\bigl\{\eta\in AC([0,T];\O):\ \dot\eta(t)=\bb(t,\eta(t))\,\,\text{for $\Leb{1}$-a.e. $t\in (0,T)$}\bigr\};
$$
\item[(ii)] there exists $C_0\in (0,\infty)$ such that
\begin{equation}\label{eqn:noconcentration}
(e_t)_\#\llambda\leq C_0\Leb{d}\qquad\forall \,t\in [0,T].
\end{equation}
\end{itemize}
Then the conditional probability measures $\llambda_x$ induced by the map $e_0$ are Dirac masses for
$(e_0)_\#\llambda$-a.e. $x$; equivalenty, there exist curves $\eta_x\in AC([0,T];\O)$ solving the Cauchy problem
$\dot\eta=\bb(t,\eta)$ with the initial condition $\eta(0)=x$, satisfying
$$
\llambda= \int \delta_{\eta_x}\,d[(e_0)_\#\llambda](x).
$$
\end{theorem}
\begin{Proof} Let $\{A_n\}_{n\in\N}$ be an increasing family of open subsets of $\O$ whose union is $\Omega$, with
$A_n\Subset A_{n+1}\Subset\O$ for every $n$.
Possibly considering the restriction of $\eeta$ to the sets 
$$
\bigl\{\eta\in C([0,T];\R^d):\ \eta(t) \in \overline{A}_n \ \text{for every $t\in[0,T]$} \bigr\}
$$
it is not restrictive to assume that $\eeta$ is concentrated on a family $\Gamma$ of curves
satisfying $\bigcup_{\eta\in\Gamma}\eta([0,T])\Subset\Omega$. Then, using the uniqueness assumption
for uniformly bounded and compactly supported solutions to the continuity equation, the result follows from the decomposition procedure of \cite[Theorem~18]{bologna} 
(notice that the latter slightly improves the original argument of \cite[Theorem~5.4]{ambrosio}, where comparison principle for the continuity equation was assumed,
see also Proposition~\ref{prop:uni_continuity} and its proof).
\end{Proof}

\begin{remark}{\rm
The assumption (\bbb-$\O$) is purely local, as it is proved in the Appendix. Moreover, it could be reformulated in terms of a local uniqueness property of regular lagrangian flows: for any $t_0\geq 0, x_0 \in \O$ there exists $\eps := \eps(t_0,x_0)>0$ such that for any Borel set $B \subset B_\eps(x_0)\subset \O$ and any closed interval $I=[a,b]\subset [t_0-\eps,t_0+\eps] \cap [0,T]$, there exists at most one regular lagrangian flow in $B \times [a,b]$ with values in $B_\eps(x_0)$ (see Definition~\ref{def:regflow}).

Indeed, (b-$\O$) implies the local uniqueness of regular lagrangian flows by Theorem~\ref{thm:nosplitting} applied to $\llambda = \frac{1}{2}\int_B(\delta_{\sXX(\cdot,x)}+ \delta_{\sYY(\cdot,x)})\,d\Leb{d}(x)$, where $\XX$ and $\YY$ are regular lagrangian flows in $B \times [a,b]$; on the other hand, we obtain the converse implication through the superposition principle.
This approach has the advantage to state the assumptions and the results of this paper only in terms of the lagrangian point of view on the continuity equation. On the other hand, in concrete examples it is usually easier to verify assumption (\bbb-$\O$) than the corresponding lagrangian formulation.
}\end{remark}

\section{Regular flow, hitting time, maximal flow}

\begin{definition} [Local regular flow] \label{def:regflow}
Let $B\in\BorelSets{\R^d}$, $\tau>0$, and  $\bb: (0,\tau) \times \R^d \to \R^d$ Borel. We say that $\XX: [0,\tau]\times B\to \R^d$ is a {\em local 
regular flow} starting from $B$ (relative
to $\bb$) up to $\tau$ if the following two properties hold: 
\begin{itemize}
\item[(i)] for $\Leb{d}$-a.e. $x\in B$, $\XX(\cdot,x)\in AC([0,\tau];\R^d)$ and solves the ODE $\dot x(t)=\bb(t,x(t))$ $\Leb{1}$-a.e.
in $(0,\tau)$, with the initial condition $\XX(0,x)=x$;
\item[(ii)] there exists a constant $C=C(\XX)$ satisfying $\XX(t,\cdot)_\#(\Leb{d}\res B)\leq C\Leb{d}$.
\end{itemize}
\end{definition}

In the previous definition, as long as the image of $[0,\tau]\times B$ through $\XX$ is contained  in an open set $\O$, it is not necessary to specify the 
vector field $\bb$ outside $\O$. By Theorem~\ref{thm:nosplitting} we obtain a consistency result of the local regular flows with values in 
$\O$ in the intersection of their domains.

\begin{lemma}[Consistency of local regular flows]\label{lem:local_coincidence}
Assume that $\bb$ satisfies (a-$\O$) and (\bbb-$\O$).
Let $\XX_i$ be local regular flows starting from $B_i$ up to $\tau_i$, $i=1,\,2$, with $\XX_i([0,\tau_i]\times B_i)\subset\O$. 
Then
\begin{equation}\label{coincidence}
\XX_1(\cdot,x)\equiv\XX_2(\cdot,x)\quad\text{in $[0,\tau_1\wedge\tau_2]$, for $\Leb{d}$-a.e. $x\in B_1\cap B_2$.}
\end{equation}
\end{lemma}
\begin{Proof} Take $B\subset B_1\cap B_2$ Borel with $\Leb{d}(B)$ finite, and apply 
Theorem~\ref{thm:nosplitting} with $T=\tau_1\wedge\tau_2$, $m=d$, and
$$
\llambda:=\frac{1}{2}\int \bigl(\delta_{\sXX_1(\cdot,x)}+\delta_{\sXX_2(\cdot,x)}\bigr)\,d\Leb{d}_B(x),
$$
where $\Leb{d}_B$ is the normalized Lebesgue measure on $B$.
\end{Proof}

If we consider a smooth vector field $\bb$ in a domain $\O$, a maximal flow of $\bb$ in 
$\O$ would be given by the trajectories of $\bb$ until they hit the boundary of $\O$.
In order to deal at the same time with bounded and unbounded domains (including the case $\O=\R^d$) we introduce a continuous potential function $V_\O:\O\to [0,\infty)$ satisfying
\begin{equation}\label{def:VOm}
\lim_{x\to\partial\O} V_\O(x)=\infty,
\end{equation}
meaning that for any $M>0$ there exists $K\Subset\O$ with $V_\O>M$ on $\O\setminus K$
(in particular, when $\Omega=\R^d$, $V_\O(x) \to \infty$ as $|x|\to \infty$).
For instance, an admissible potential is given by $ V_\O(x)= \max\{ [{\rm dist}(x, \R^d \setminus \O)]^{-1}, |x|\}$.

\begin{definition}[Hitting time in $\O$]\label{defn:exist-time}
Let $\tau>0$, $\O\subset\R^d$ open and $\eta: [0,\tau)\to \R^d$ continuous. 
We define the {\em hitting time} of $\eta$ in $\O$ as
$$
\hit{\O}{\eta}: = \sup\{ t\in [0,\tau):\ \max_{[0,t]}V_\O(\eta)<\infty\},
$$
with the convention $\hit{\O}{\eta}=0$ if $\eta(0)\notin\O$.
\end{definition}

It is easily seen that this definition is independent of the choice of $V_\O$, that $\hit{\O}{\eta}>0$ whenever
$\eta(0)\in\O$, and that
\begin{equation}\label{eq:blowuptimeeta}
\hit{\O}{\eta}<\tau\quad\Longrightarrow\quad\limsup_{t\uparrow\hit{\O}{\eta}}V_\O(\eta(t))=\infty.
\end{equation}

Using $V_\O$ we can also define the concept of maximal regular flow, where ``regular''
refers to the local bounded compression condition \eqref{eqn:incompr-o}.

\begin{definition} [Maximal regular flow in an open set $\O$]\label{def:maxflow-o} 
Let  $\bb: (0,T) \times \O \to \R^d$ be a Borel vector field. We say that a Borel map $\XX$
is a {\em maximal regular flow} relative to $\bb$ in $\O$ if there exists a Borel map
$T_{\O,\sXX}:\O\to (0,T]$ such that $\XX(t,x)$ is defined in the set $\{(t,x):\ t<T_{\O,\sXX}(x)\}$ and
 the following properties hold:
\begin{itemize}
\item[(i)] for $\Leb{d}$-a.e. $x\in\O$, $\XX(\cdot,x)\in {AC_{\rm loc}([0,T_{\sox}(x));\R^d)}$,  and solves the ODE $\dot x(t)=\bb(t,x(t))$ $\Leb{1}$-a.e.
in $(0,T_{\sox}(x))$, with the initial condition $\XX(x,0)=x$;
\item[(ii)] for any $\O'\Subset\O$ there exists a constant $C(\O',\XX)$ such that 
\begin{equation}
\label{eqn:incompr-o}
\XX(t,\cdot)_\#(\Leb{d}\res\{T_{\O'}>t\})\leq C(\O',\XX) \Leb{d}\res\O'\qquad \textit{for
all $t\in [0,T]$,}
\end{equation}
where
$$
T_{\O'}(x):=\left\{
\begin{array}{ll}
\hit{\O'}{\XX(\cdot,x)} &\text{ for $x\in\O'$},\\
0 & \text{ otherwise;}
\end{array}
\right.
$$ 
\item[(iii)] $\limsup\limits_{t\uparrow T_{\O,\ssXX}(x)}V_\O(\XX(t,x))=\infty$ for $\Leb{d}$-a.e. $x\in\O$ such that $T_{\sox}(x)<T$.
\end{itemize}
\end{definition}

Notice that \eqref{eqn:incompr-o} could be equivalently written as
$$
\XX(t,\cdot)_\#(\Leb{d}\res\{T_{\O'}>t\})\leq C(\O',\XX) \Leb{d}\qquad \text{for
all $t\in [0,T]$,}
$$
because the push-forward measure is concentrated on $\O'$; so the real meaning of this requirement
is that the push forward measure must have a bounded density w.r.t. $\Leb{d}$.

\begin{remark}[Maximal regular flows induce regular flows]\label{rem:induce_local-o}
{\rm Given any maximal regular flow $\XX$ in $\O$, $\tau\in (0,T)$, and a Borel set $B\subset\O$ such that
$T_{\sox}>\tau$ on $B$ and
$$
\bigl\{\XX(t,x):\ x\in B,\,t\in [0,\tau]\bigr\}\Subset\O,
$$ 
we have an induced local regular flow in the set $B$ up to time $\tau$.
}\end{remark}

\begin{remark} [Invariance in the equivalence class of $\bb$]\label{rem:invaria_flow}
{\rm It is important and technically useful (see for instance \cite{cetraro}) to underline that the concepts of local regular flow and of maximal regular
flow are invariant in the Lebesgue equivalent class, exactly as our constitutive assumptions (a-$\O$), (\bbb-$\O$), and the global/local bounds on the divergence of $\bb$. Indeed, for local regular flows, Definition~\ref{def:regflow}(ii) 
in conjunction with Fubini's theorem implies that for any $\Leb{1+d}$-negligible set $N\subset (0,T)\times\R^d$ the set
$$
\left\{x\in B:\ \Leb{1}(\{t\in (0,\tau):\ (t,\XX(t,x))\in N\})>0\right\}
$$
is $\Leb{d}$-negligible. An analogous argument, based on \eqref{eqn:incompr-o}, applies to maximal regular flows. 
}\end{remark}

\section{Existence and uniqueness of the maximal regular flow}\label{sec:local}

In this section
we consider a Borel vector field $\bb: (0,T) \times \O \to \R^d$ which satisfies the assumptions (a-$\O$), (\bbb-$\O$) of Section~\ref{sec:ass-a-c},
and such that the spatial divergence ${\rm div\,}\bb(t,\cdot)$ in the sense of distributions satisfies 
\begin{equation}\label{eqn:incompb-o-L}
\forall\,\O'\Subset\O,\,\,\,{\rm div\,}\bb(t,\cdot)\geq m(t)\quad\text{in $\O'$, with $L(\O',\bb):=\int_0^T |m(t)|\, dt<\infty$.}
\end{equation}

\begin{remark}\label{rmk:manycases-bouchut-crippa} {\rm Assumption \eqref{eqn:incompb-o-L} could be weakened to $m\in L^1(0,T_0)$ for all
$T_0\in (0,T)$, but we made it global in time to avoid time-dependent constants in our estimates (and, in any case, the maximal
flow could be obtained in this latter case by a simple gluing procedure w.r.t. time).
}\end{remark}

The first step in the construction of the maximal regular flow will be the following local existence result.

\begin{theorem}[Local existence]\label{thm:auxiliary} Let $\bb: (0,T) \times \O \to \R^d$ be a Borel vector field which satisfies (a-$\O$),  (\bbb-$\O$),
\eqref{eqn:incompb-o-L}, and let $A\Subset\O$ be open. Then there exist a Borel map $T_A:A\to (0,T]$ and a Borel map
$\XX(t,x)$, defined for $x\in A$ and $t\in [0,T_A(x)]$, such that:
 \begin{itemize}
\item[(a)] for $\Leb{d}$-a.e. $x\in A$, $\XX(\cdot,x)\in AC([0,T_A(x)];\R^d)$, $\XX(0,x)=x$, $\XX(t,x)\in A$ for all
$t\in [0,T_A(x))$, and $\XX(T_A(x),x)\in\partial A$ when $T_A(x)<T$;
\item[(b)] for $\Leb{d}$-a.e. $x\in A$, $\XX(\cdot,x)$ solves the ODE $\dot\gamma=\bb(t,\gamma)$ in $(0,T_A(x))$;
\item[(c)] $\XX(t,\cdot)_\#(\Leb{d}\res\{T_A>t\})\leq e^{L(A,\sbb)}\Leb{d}\res A$ for all $t\in [0,T]$, where $L(A,\bb)$ is
the constant in \eqref{eqn:incompb-o-L}.
\end{itemize}
\end{theorem}

Notice that since the statement of the theorem is local (see also the appendix, in connection with
property (\bbb-$\O$)), we need only to prove it under the assumption
$|\bb|\in L^1((0,T)\times\O)$, which is stronger than (a-$\O$).

We will obtain Theorem~\ref{thm:auxiliary} via an approximation procedure which involves the concept of 
regular generalized flow in closed domains, where now ``regular'' refers to the fact that the bounded compression condition is imposed
only in the interior of the domain.

\begin{definition}[Regular generalized flow in $\overline{A}$]\label{def:dregbar}
Let  $A\subset\R^d$ be an open set and let $\cc: (0,T) \times\overline{A}\to \R^d$ be a Borel vector field. A probability 
measure $\eeta$ in $C([0,T];\R^d)$ is said to be a {\em regular generalized flow} in $\overline{A}$ if the following two conditions hold:
\begin{itemize}
\item[(i)] $\eeta$ is concentrated on
$$
\bigl\{\eta\in AC([0,T];\overline{A}):\ \dot\eta(t)=\cc(t,\eta(t))\,\,\text{for $\Leb{1}$-a.e. $t\in (0,T)$}\bigr\};
$$
\item[(ii)] there exists $C:= C(\eeta)\in (0,\infty)$ satisfying
\begin{equation}\label{eqn:noconcentration-o}
((e_t)_\#\eeta) \res A \leq C\Leb{d}\qquad\forall \,t\in [0,T].
\end{equation}
\end{itemize}
\end{definition}
Any constant $C$ for which \eqref{eqn:noconcentration-o} holds is called a {\em compressibility constant} of $\eeta$.

The class of regular generalized flows enjoys good tightness and stability properties. We state them in the
case of interest for us, namely when the velocity vanishes at the boundary.

\begin{theorem}[Tightness and stability of regular generalized flows in $\overline{A}$]\label{thm:tightness-o}
Let $A\subset\R^d$ be a bounded open set, let $\cc,\,\cc^n : (0,T) \times \overline{A} \to \R^d$ be Borel vector fields such that 
$\cc=\cc^n=0$ on $(0,T) \times (\R^d \setminus A)$ and
\begin{equation}
\label{eqn:bn-conv}
\lim_{n\to\infty} \cc^n=\cc \qquad \textit{in }L^1((0,T) \times A;\R^d).
\end{equation} 
Let  $\eeta^n\in\Probabilities{C([0,T];\overline{A})}$ be regular generalized flows of $\cc^n$ in $\overline{A}$ and let us assume that
the best compressibility constants $C_n$ of $\eeta^n$ satisfy $\sup_n C_n <\infty$.
Then $(\eeta^n)$ is tight, any limit point $\eeta$ is a regular generalized  flow of $\cc$ in $\overline{A}$, and the following implication holds:
\begin{equation}\label{eq:pasquetta4}
((e_t)_\#\eeta^n\res\Gamma) \res A'\leq c_n\Leb{d} \quad \text{for some $c_n>0$}\qquad\Longrightarrow\quad
((e_t)_\#\eeta\res\Gamma) \res A'\leq \liminf_n c_n\Leb{d}
\end{equation}
for any choice of open sets $\Gamma\subset C([0,T];\overline{A})$ and $A'\subset A$.
\end{theorem}
\begin{Proof}
By Dunford-Pettis' theorem, since the family $\{\cc^n\}$ is compact in $L^1(\overline{A};\R^d)$ (recall that $\cc_n(t,\cdot)$ vanish outside of $A$), there exists a modulus of integrability 
for $\cc^n$, namely an increasing, convex, 
superlinear function $F : [0,\infty) \to [0,\infty)$ such that $F(0) = 0$ and
\begin{equation}\label{eqn:high-int}
\sup_{n\in \N} \int_0^T\int_{\overline{A}} F(|\cc^n(t,x)|) \, dxdt <\infty.    
\end{equation}

Let us introduce the functional $\Sigma:C([0,T];\R^d)\to [0,\infty]$ as follows
$$
\Sigma(\eta):=\left\{
\begin{array}{ll}
\int_0^TF(|\dot\eta(t)|)\,dt & \text{if $\eta\in AC([0,T];\overline{A})$,}\\
\infty & \text{if $\eta\in C([0,T];\R^d)\setminus AC([0,T];\overline{A})$.}
\end{array}
\right.
$$
Using Ascoli-Arzel\`a theorem, the compactness of $\overline{A}$, and a well-known lower semicontinuity result 
due to Ioffe (see
for instance \cite[Theorem~5.8]{afp}), it turns out that $\Sigma$ is lower 
semicontinuous and coercive, namely its sublevels $\{\Sigma\leq M\}$ are compact. 

Since $\eeta^n$ is concentrated on $AC([0,T];\overline{A})$ we get
$$
\int\Sigma\,d\eeta^n=
\int\!\int_0^TF(|\dot\eta|)\,dt\,d\eeta^n(\eta)=
\int_0^T\!\int_{\overline{A}}F(|\cc^n|)\,d[(e_t)_\#\eeta^n] \,dt
\leq C_n
\int_0^T\!\int_{\overline{A}}F(|\cc^n|)\,dx\,dt,
$$
so that that $\int\Sigma\,d\eeta^n$ is uniformly bounded thanks to \eqref{eqn:high-int}.
Therefore Prokhorov compactness theorem provides the existence of limit points.
Since $\Sigma$ is lower semicontinuous we obtain that any limit point $\eeta$
satisfies $\int\Sigma\,d\eeta<\infty$, therefore
$\eeta$ is concentrated on $AC([0,T];\overline{A})$.

Let $C := \liminf_{n\in \N} C_n <\infty$. Since $(e_t)_\#\eeta^n$ narrowly converge to $(e_t)_\#\eeta$, we know that
for any open set $A'\subset A$ there holds 
$$(e_t)_\#\eeta (A') \leq \liminf_{n\to \infty} (e_t)_\#\eeta^n(A') \leq C\Leb{d}(A')\qquad\forall \,t\in [0,T].
$$
Since $A'$ is arbitrary we deduce that $\eeta$ satisfies \eqref{eqn:noconcentration-o}. A similar argument
provides its localized version \eqref{eq:pasquetta4}.
 To show that $\eeta$ is concentrated on integral curves of $\cc$, it suffices to show that
\begin{equation}\label{valleggi}
\int\left|\eta(t)-\eta(0)-\int_0^t\cc(s,\eta(s))\,ds\right|\,d\eeta(\eta)=0
\end{equation}
for any $t\in [0,T]$. The technical difficulty is that this test function,
due to the lack of regularity of $\cc$, is not continuous. 
To this aim, we prove that
\begin{equation}\label{valleggi1}
\int\left|\eta(t)-\eta(0)-\int_0^t\cc'(s,\eta(s))\,ds\right|\,d\eeta(\eta)\leq C
\int_{(0,T)\times A}|\cc-\cc'|\,dx \, dt
\end{equation}
for any continuous vector field $\cc':[0,T]\times\overline{A}\to\R^d$ with $\cc'=0$ in $[0,T]\times\partial A$. 
Then, choosing a sequence $(\cc_n')$ of such vector fields
converging to $\cc$ in $L^1(\overline{A};\R^d)$ and noticing that
$$
\int\!\int_0^T|\cc(s,\eta(s))-\cc_n'(s,\eta(s))|\,dsd\eeta(\eta)
=\int_0^T\!\!\int_{A}|\cc-\cc_n'|\,d(e_s)_\# \eeta \,ds\leq C_n\int_{(0,T)\times A}\!\!\!|\cc-\cc_n'|\,dx\, dt,
$$
converges to $0$ as $n$ goes to $\infty$,
we can take the limit in \eqref{valleggi1} with $\cc'=\cc_n'$ to obtain \eqref{valleggi}.

It remains to show \eqref{valleggi1}. This is a limiting argument based on the fact
that \eqref{valleggi} holds for $\cc^n$, $\eeta^n$:
\begin{eqnarray*}
\int\left|\eta(t)-\eta(0)-\int_0^t\cc'(s,\eta(s))\,ds\right|\,d\eeta^n(\eta)&=&
\int\left|\int_0^t\bigl(\cc^n(s,\eta(s)-\cc'(s,\eta(s))\bigr)\,ds\right|\,d\eeta^n(\eta)\\
&\leq&
\int\int_0^t|\cc^n-\cc'|(s,\eta(s))\,ds \, d\eeta^n(\eta)
\\
&=&
\int_0^t\int_{A}|\cc^n-\cc'|\,d[(e_s)_\#\eeta^n] \,ds
\\
&\leq& C_n
\int_0^t\int_{A}|\cc^n-\cc'|\,dx\, ds.
\end{eqnarray*}
Taking the limit in the chain of inequalities above we obtain \eqref{valleggi1}.
\end{Proof}

Now we show how Theorem~\ref{thm:auxiliary} can be deduced from the existence of regular generalized flows in
$\overline{A}$; at the same time, we show that flows associated to sufficiently smooth vector fields induce
regular generalized flows (actually even classical ones, but we will need them in generalized form to take limits).

\begin{proposition}\label{prop:auxiliary}
(i) Let $\bb: (0,T) \times \O \to \R^d$ be a Borel vector field which satisfies (a-$\O$) and  (\bbb-$\O$), 
let $A\Subset\O$ be an open set, and
let $\eeta$ be a regular generalized flow in $\overline{A}$ relative to $\cc=\chi_A\bb$ with compressibility constant $C$ and that satisfies
$(e_0)_\#\eeta=\rho_0\Leb{d}$ with $\rho_0>0$ $\Leb{d}$-a.e. in $A$.
Then there exist $\XX$ and $T_A$ as in Theorem~\ref{thm:auxiliary}(a)-(b) that satisfy
\begin{equation}
\label{eqn:compr-rho0}
\XX(t,\cdot)_\# (\rho_0\res\{T_A>t\})\leq C\Leb{d}\res A
\end{equation}
 for all $t\in [0,T]$.

(ii) Let $\bb\in C^\infty([0,T]\times\overline{A};\R^d)$. Then there exists a regular generalized
flow $\eeta$ associated to $\bb\chi_A$, with $(e_0)_\#\eeta$ equal to the normalized Lebesgue measure in $A$ 
and satisfying
\begin{equation}\label{eq:pasquetta5}
((e_t)_\#\eeta\res \{\hit{A'}{\cdot}>t\}) \res A'\leq \frac{e^{L(A',\sbb)}}{\Leb{d}(A)}\Leb{d}\qquad\forall \,t\in [0,T]
\end{equation}
for any open set $A'\Subset A$.
\end{proposition}
\begin{Proof} We first prove (i). Set $\mu_0=\rho_0\Leb{d}$ and consider a family $\{\eeta_x\}\subset\Probabilities{C([0,T];\overline{A})}$
of conditional probability measures, concentrated on 
$$
\bigl\{\eta\in AC([0,T];\overline{A}):\ \text{$\dot\eta=\cc(t,\eta)$ $\Leb{1}$-a.e. in $(0,T)$, $\eta(0)=x$}\bigr\}
$$
and representing $\eeta$, i.e., $\int \eeta_x\,d\mu_0(x)=\eeta$. We claim that $\mu_0$-almost every $x\in A$:
\begin{enumerate}
\item[(1)]
$\hit{A}{\eta}$ is equal to a positive constant for $\eeta_x$-a.e. $\eta$;
\item[(2)]
if $T_A(x)$ is the constant in (1), $(e_t)_\#\eeta_x$ is a Dirac mass for all $t\in [0,T_A(x)]$.
\end{enumerate}
By our assumption on $\mu_0$, the properties stated in the claim hold $\Leb{d}$-a.e. in $A$. Hence,
given the claim, if we define 
$$
\XX(t,x):=\int \eta(t)\,d\eeta_x(\eta)
$$
then for $\Leb{d}$-a.e. $x\in A$ the integrand $\eta(t)$ is independent of $\eta$ as soon as $t<T_A(x)$,
hence $\XX(t,x)$ satisfies (a) and (b) in the statement of Theorem~\ref{thm:auxiliary}. The compressibility property \eqref{eqn:compr-rho0} follows
immediately from \eqref{eqn:noconcentration-o}.

Let us prove our claim. We notice that the hitting time is positive for $\mu_0$-a.e. $x\in A$. 
For $q\in\Q\cap (0,T)$, we shall denote by $\Gamma_q$ the set $\{\eta:\ \hit{A}{\eta}>q\}$ and by
$\Sigma^q:\Gamma_q\to C([0,q];A)$ the map induced by restriction to $[0,q]$, namely $\Sigma^q(\eta)=\eta\vert_{[0,q]}$.

In order to prove the claim it clearly suffices to 
show that, for all $q\in\Q\cap (0,T)$, $\Sigma^q_\#(\eeta_x\res\Gamma_q)$ is either a Dirac mass or it is null.
So, for $q\in\Q\cap (0,T)$ and $\delta\in (0,1)$ fixed, it suffices to show that
$$
\llambda_x:=\frac{1}{\eeta_x(\Gamma_q)}\Sigma_\#^q(\eeta_x\res\Gamma_q)\in\Probabilities{C([0,q];A)}
$$
is a Dirac mass for $\mu_0$-a.e. $x$ satisfying $\eeta_x(\Gamma_q)\geq\delta$. 

By construction the measures $\llambda_x$ satisfy $\llambda_x\leq\Sigma_\#^q(\eeta_y\res\Gamma_q)/\delta$ 
and they are concentrated on curves $[0,q]\ni t \mapsto \eta(t)$ starting at $x$ and solving the
ODE $\dot\eta=\bb(t,\eta)$ in $(0,q)$. Therefore
$$
\llambda:=\int_{\{x\in A:\ \seeta_{x}(\Gamma_q)\geq\delta\}}
\llambda_{x}\,d\mu_0(x)
\in\Probabilities{C([0,q];A)}
$$
satisfies all the assumptions of Theorem~\ref{thm:nosplitting} with $T=q$ and $\O=A$, provided we check \eqref{eqn:noconcentration}.
To check this property with $C_0=C/\delta$, for $t\in [0,q]$ and $\varphi\in C_c(A)$ nonnegative we use the fact that $\llambda_y \leq \Sigma_\#^q(\eeta_y\res\Gamma_q)/\delta$ and the fact that $C$ is a compressibility constant of $\eeta$ to estimate
$$
\int_{\R^d}\varphi\,d(e_t)_\#\llambda\leq
\frac 1\delta\int_{\R^d}\varphi\,d(e_t)_\#(\eeta \res \Gamma_q)
\leq\frac 1\delta\int_{\R^d}\varphi\,d(e_t)_\#\eeta
\leq \frac{C}{\delta} \int_{A}\varphi\, dx.
$$
Therefore Theorem~\ref{thm:nosplitting}
can be invoked: $\llambda_{x}$ is a Dirac mass for $\mu_0$-a.e. $x$ and this gives
that $\llambda_x$ is a Dirac mass $\mu_0$-a.e. in $\{\eeta_x(\Gamma_q)\geq\delta\}$.
This concludes the proof of (i).

For (ii), we begin by defining $\eeta$ with the standard Cauchy-Lipschitz theory. More precisely,
for
$x\in A$ we  let $\XX(t,x)$ be the unique solution to the ODE $\dot\eta=\bb(t,\eta)$
with $\eta(0)=x$ until the first time
$T_A(x)$ that $\XX(t,x)$ hits $\partial A$, and then we define $\XX(t,x)=\XX(t,T_A(x))$ for all $t\in [T_A(x),T]$. Finally,
denoting by $\Leb{d}_A$ the normalized Lebesgue measure in $A$, we define $\eeta$ as the law under
$\Leb{d}_A$ of the map $x\mapsto \XX(\cdot,x)$. With this
construction it is clear that condition (i) in Definition~\ref{def:dregbar} holds. 

Let us check condition (ii) as
well, in the stronger form \eqref{eq:pasquetta5}. Recall that $\XX$ is smooth before the hitting time and that the map
$t\mapsto J(t):={\rm det\,}\nabla_x\XX(t,x)$ is nonnegative and solves the ODE
\begin{equation}\label{ODEJ}
\left\{
\begin{array}{l}
 \dot J(t)=J(t)\,{\rm div\,}\bb(t,\XX(t,x)),\\
J(0)=1.
\end{array}
\right.
\end{equation}
Now, fix an open set $A'\Subset A$, and observe that \eqref{eq:pasquetta5} is equivalent to prove that for every $t\in [0,T]$
$$
\int_{A' \cap \{x: \hit{A'}{\sXX(\cdot,x)}>t\}} \varphi(\XX(t,x)) \, dx \leq 
e^{L(A',\sbb)} \int_{A'} \varphi(x) \, dx 
\qquad \mbox{for every } \varphi\in C_c(A').
$$
Fix $\varphi\in C_c(A')$ nonnegative and
notice that $\varphi(\XX(t,x))=0$ if $t\geq\hit{A'}{\XX(\cdot,x)}$, hence ${\rm supp\,}\varphi\circ\XX(t,\cdot)$ is a compact subset of the open
set $G_t:=\{x:\ \hit{A'}{\XX(\cdot,x)}>t\}$. By the change 
of variables formula
$$
\int_{\R^d}\varphi(\XX(t,x))\,{\rm det\,}\nabla_x\XX(t,x)\,dx
=
\int_{\R^d}\varphi(x)\,dx,
$$
in order to estimate from below the left-hand side it suffices to estimate from below ${\rm det\,}\nabla_x\XX(t,x)$
in $G_t$; using \eqref{ODEJ} and Gronwall's lemma, this estimate is provided by $e^{-L(A',\sbb)}$.
\end{Proof}

\begin{remark} \label{rem:auxiliary}{\rm
For the proof of Theorem~\ref{thm:stab-o} we record the following facts, proved but not stated in Proposition~\ref{prop:auxiliary}: if
$\eeta$ is as in the statement of the proposition, then for $(e_0)_\eeta$-a.e. $x$ the hitting time $\hit{A}{\eta}$ is equal to a positive constant 
$T_A(x)$ for $\eeta_x$-a.e. $\eta$; furthermore, $(e_t)_\#\eeta_x$ is a Dirac mass for all $t\in [0,T_A(x)]$.}
\end{remark}

\begin{proof}[Proof of Theorem~\ref{thm:auxiliary}] By the first part of Proposition~\ref{prop:auxiliary}, it suffices to build 
a regular generalized flow $\eeta$ in $\overline{A}$ relative to $\cc=\chi_A\bb$ with compressibility constant $e^{L(A, \sbb)}/\Leb{d}(A)$ such that
$(e_0)_\#\eeta=\rho_0\Leb{d}$ with $\rho_0>0$ $\Leb{d}$-a.e. in $A$. By the second part of the proposition, 
we have  existence of $\eeta$ with $(e_0)_\#\eeta$ equal to the normalized Lebesgue measure $\Leb{d}_A$
and satisfying \eqref{eq:pasquetta5} whenever $\bb\in C^\infty([0,T]\times\overline{A};\R^d)$. 

Hence, to use this fact, extend $\bb$ with the 0 value to $\R\times\R^d$ and let $\bb_\eps$ be mollified vector fields. We have that
$L(A,\bb_\eps)$ are uniformly bounded (because $A\Subset\O$) and, in addition, the properties of convolution
immediately yield
\begin{equation}\label{eq:pasquetta6}
\limsup_{\eps\downarrow 0}L(A',\bb_\eps)\leq L(A,\bb)
\qquad\text{for any $A'\Subset A$ open.}
\end{equation}
If $\eeta_\eps$ are regular generalized flows associated to $\cc_\eps=\chi_A\bb_\eps$, we can apply
Theorem~\ref{thm:tightness-o} to get that any limit point $\eeta$ is a regular generalized flow associated to $\cc$ and it
satisfies $(e_0)_\#\eeta=\Leb{d}_A$. In addition, given $A'\Subset A$ open we have
$$
((e_t)_\#\eeta_\eps\res \{\hit{A'}{\cdot}>t\}) \res A'\leq  \frac{e^{L(A',\sbb_\eps)}}{\Leb{d}(A)}\Leb{d}\qquad\forall \,t\in [0,T],
$$
thus \eqref{eq:pasquetta4} and \eqref{eq:pasquetta6} yield
$$
((e_t)_\#\eeta\res \{\hit{A'}{\cdot}>t\})\res A'\leq \frac{e^{L(A,\sbb)}}{\Leb{d}(A)}\Leb{d}\qquad\forall \,t\in [0,T].
$$
Letting $A'\uparrow A$ gives that $e^{L(A,\sbb)}/\Leb{d}(A)$ is a compressibility constant
for $\eeta$.
\end{proof}

Using a guing procedure in space, we can now build the maximal regular flow in $\Omega$ using the flows
provided by Theorem~\ref{thm:auxiliary}
in domains $\Omega_n\Subset\Omega_{n+1}$ with $\Omega_n\uparrow\Omega$.

\begin{theorem}\label{thm:main-o} Let $\bb: (0,T) \times \O \to \R^d$ be a Borel vector field which satisfies (a-$\O$) and (\bbb-$\O$).
Then the maximal regular flow is unique, and existence is ensured under the additional assumption \eqref{eqn:incompb-o-L}.
In addition, 
\begin{itemize}
\item[(a)]  for any $\O'\Subset\O$ the compressibility constant $C(\O',\XX)$ in Definition~\ref{def:maxflow-o} can be taken to be $e^{L(\O',\sbb)}$, 
where $L(\O',\bb)$ is the constant in \eqref{eqn:incompb-o-L};
\item[(b)] if  $\YY$ is a regular flow in $B$ up to $\tau$ with values in $\O$, then $T_{\O,\sXX}>\tau$ $\Leb{d}$-a.e. in $B$ and
\begin{equation}\label{coincidence1-o}
\XX(\cdot,x)=\YY(\cdot,x)\quad\text{in $[0,\tau]$, for $\Leb{d}$-a.e. $x\in B$}.
\end{equation}
\end{itemize}
\end{theorem}

\begin{Proof}
 Let us prove first the uniqueness of the maximal regular flow in $\O$. Given regular maximal flows
$\XX^i$ in $\O$, $i=1,\,2$, by Lemma~\ref{lem:local_coincidence} and Remark~\ref{rem:induce_local-o} we easily obtain
$$
\XX^1(\cdot,x)=\XX^2(\cdot,x)\quad\text{in $[0,T_{\O, \sXX^1}(x)\wedge T_{\O,\sXX^2}(x))$, for $\Leb{d}$-a.e. $x\in\O$}.
$$
On the other hand, for $\Leb{d}$-a.e. $x \in \{T_{\O,\sXX^1}>T_{\O,\sXX^2}\}$, the image of $[0,T_{\O,\sXX^2}(x)]$ through $V_\O(\XX^1(\cdot,x))$ is bounded in $\R$, whereas the image of $[0,T_{\O,\sXX^2}(x))$ through $V_\O(\XX^2(\cdot,x))$ is not.
It follows that the set
$\{T_{\sXX^1}>T_{\sXX^2}\}$ is $\Leb{d}$-negligible. Reversing the roles of $\XX^1$ and $\XX^2$ we obtain that
$T_{\O,\sXX^1}=T_{\O,\sXX^2}$ $\Leb{d}$-a.e. in $\O$.

In order to show existence we are going to use auxiliary flows $\XX_n$ in $\O_n$ with hitting times $T_n:\O_n\to (0,T]$, i.e.,
\begin{itemize}
\item[(1)] for $\Leb{d}$-a.e. $x\in\O_n$, $\XX_n(\cdot,x)\in AC([0,T_n(x)];\R^d)$, $\XX_n(0,x)=x$, $\XX_n(t,x)\in\O_n$ for all
$t\in [0,T_n(x))$, and $\XX_n(T_n(x),x)\in\partial\O_n$ when $T_n(x)<T$, so that $\hit{\O_n}{\XX_n(\cdot,x)}=T_n(x)$;
\item[(2)] for $\Leb{d}$-a.e. $x\in\O_n$, $\XX_n(\cdot,x)$ solves the ODE $\dot\gamma=\bb(t,\gamma)$ in $(0,T_n(x))$;
\item[(3)] $\XX_n(t,\cdot)_\#(\Leb{d}\res\{T_n>t\})\leq e^{L(\O_n,\sbb)}\Leb{d}\res\O_n$ for all $t\in [0,T]$, where $L(\O_n,\bb)$ is
given as in \eqref{eqn:incompb-o-L}.
\end{itemize}
The existence of $\XX_n$, $T_n$ as in (1), (2), (3) has been achieved in Theorem~\ref{thm:auxiliary}.  

If $n\leq m$, the uniqueness argument outlined at the beginning of this proof gives immediately that $T_n(x)\leq T_m(x)$, and that $\XX_n(\cdot,x)\equiv \XX_m(\cdot,x)$ in $[0,T_n(x)]$ for $\Leb{d}$-a.e. $x\in\O_n$. Hence the limits
\begin{equation}
\label{eqn:times-limit-o}
T_{\sox}(x) := \lim_{n\to \infty} T_n(x), \qquad
\XX(t,x) =  \lim_{n\to \infty} \XX_n(t,x) \quad t\in [0,T_{\sox}(x))
\end{equation}
are well defined for $\Leb{d}$-a.e. $x\in\O$. By construction
\begin{equation}
\label{eqn:X-Xn}
\XX(\cdot,x) = \XX_n(\cdot,x) \qquad \text{in $[0,T_n(x))$, for $\Leb{d}$-a.e. $x\in\O_n$}.
\end{equation}

We now check that $\XX$ and $T_{\sox}$ satisfy the conditions (i), (ii), (iii) of Definition~\ref{def:maxflow-o}.
Property (i) is a direct consequence of property (2) of $\XX_n$, \eqref{eqn:times-limit-o}, and \eqref{eqn:X-Xn}. 

In connection with property (ii) of Definition~\ref{def:maxflow-o}, in the more specific form stated in (a)
for any open set $\O'\Subset\O$, it suffices to check it for all open sets $\O_n$: indeed, it is clear that 
in the uniqueness proof we need it only for a family of sets that invade $\O$ and, as soon as uniqueness
is estabilished, we can always assume in our construction that $\O'$ is one of the sets $\O_n$. Now,
given $n$, we first remark that property (1) of $\XX_n$ yields $T_n(x)=\hit{\O_n}{\XX(\cdot,x)}$ for $\Leb{d}$-a.e.
$x\in\O_n$; moreover
\eqref{eqn:X-Xn} gives
$$\XX(t,\cdot)_\# (\Leb{d}\res \{T_n>t\}) = \XX_n(t,\cdot)_\#(\Leb{d} \res \{T_n>t\})$$
for all $t\in [0,T]$. Hence, we can now use property (3) of $\XX_n$ to get
\begin{equation}
\label{eqn:incompr-in-o}
\XX(t,\cdot)_\#(\Leb{d}\res \{T_n>t\})\leq e^{L(\O_n,\sbb)}\Leb{d}\res\O_n\qquad \text{for every }t\in [0,T],
\end{equation}
which together with the identity $T_n(x)=\hit{\O_n}{\XX(\cdot,x)}$ for $\Leb{d}$-a.e.
$x\in\O_n$ concludes the verification of Definition~\ref{def:maxflow-o}(ii).

Now we check Definition~\ref{def:maxflow-o}(iii): we obtain that $\limsup V_\O(\XX(t,x))=\infty$
 as $t\uparrow T_{\sox}(x)$ for $\Leb{d}$-a.e. $x\in\O$ such that $T_{\sox}(x)<T$
from the fact that $\XX(t,T_n(x))\in\partial\Omega_n$, and the sets $\O_n$ contain eventually any set $K\Subset\O$.
This completes the existence proof and the verification of the more specific property (a).

The proof of property (b) in the statement of the theorem follows at once from Lemma~\ref{lem:local_coincidence} and Remark~\ref{rem:induce_local-o}.
\end{Proof}

\section{Main properties of the maximal regular flow}

\subsection{Semigroup property}
In order to discuss the semigroup property, we double the time variable and denote by
$$
\XX(t,s,x),\qquad t\geq s,
$$
the maximal flow with $s$ as initial time, so that $\XX(t,0,x)=\XX(t,x)$ and $\XX(s,s,x)=x$.
The maximal time of $\XX(\cdot,s,x)$ will be denoted by $T_{\O, \sXX ,s}(x)$.

The proof of the semigroup property and of the identity
$T_{\O,\sXX,s}(\XX(s,x))=T_{\O,\sXX}(x)-s$ satisfied by
the maximal existence time follows the classical scheme. It is however a bit more involved than usual
because we are assuming only one-sided bounds on the divergence of $\bb$, therefore the inverse of the map $\XX(s,\cdot)$
(which corresponds to a flow with reversed time) is a priori not defined. For this reason, using disintegrations, we
define in the proof a kind of multi-valued inverse of $\XX(s,\cdot)$.

\begin{theorem}[Semigroup property]\label{thm:semigr}
Under assumptions (a-$\O$), (\bbb-$\O$), and \eqref{eqn:incompb-o-L} on $\bb$, for all $s \in [ 0, T]$ the maximal regular flow $\XX$ satisfies
\begin{equation}\label{eqn:consecutive-o}
T_{\O,\sXX,s}(\XX(s,x)) = T_{\O, \sXX}(x)-s
\qquad\text{for $\Leb{d}$-a.e. $x\in\{T_{\O,\sXX}>s\}$,}
\end{equation}
\begin{equation}\label{eqn:semig3-o-int}
\XX\bigl(\cdot ,s,\XX(s,x)\bigr)=\XX(\cdot+s,x) \text{ in $[0,T_{\O, \sXX}(x)-s)$,  }
\qquad\text{for $\Leb{d}$-a.e. $x\in\{T_{ \O, \sXX}>s\}$.}
\end{equation}
\end{theorem}
\begin{Proof} Let us fix $s\geq 0$ and assume without loss of generality that $\Leb{d}(\{T_{\O,\sXX}>s\})>0$.
Let us fix a Borel $B_s\subset\{T_{\O,\sXX}>s\}$ with positive and finite measure,
and let $\Leb{d}_s $ denote the renormalized Lebesgue measure on $B_s$, namely $\Leb{d}_s := \Leb{d} \res B_s / \Leb{d}(B_s)$. We denote by 
$\rho_s$ the bounded density of the probability measure $\XX(s,\cdot)_\#\Leb{d}_s$ with respect to $\Leb{d}$.
We can disintegrate the probability measure $\pi:=(Id\times\XX(s,\cdot))_\#\Leb{d}_s$ with respect to $\rho_s$, getting a family $\{\pi_y\}$ of probability measures
in $\R^d$ such that $\pi=\int \pi_y\otimes\delta_y\, \rho_s(y)\,dy$. Notice that in the case when $\XX(s,\cdot)$ is (essentially) injective,
$\pi_y$ is the Dirac mass at $(\XX(s,\cdot))^{-1}(y)$ for $\XX(s,\cdot)_\#\Leb{d}_s $-a.e. $y$. 

For $\eps>0$, let us set
$$
\pi_\eps:=\int_{\{\rho_s\geq\eps\}} \pi_y\otimes\delta_y \,dy \in \Probabilities{\R^{2d}}
$$
Since $\eps\pi_\eps\leq \pi$,  the first marginal $\tilde \rho_\eps$ of $\pi_\eps$ is bounded from above by $\Leb{d}_s /\eps$,
therefore it has a bounded density $\tilde\rho_\eps$ with respect to $\Leb{d}$. Moreover, since
$\pi\leq\|\rho_s\|_{L^\infty(\R^d)}\sup_{\eps>0}\pi_\eps$ and the first marginal of $\pi$ is $\Leb{d}_s$, we obtain
\begin{equation}\label{eqn:nohhole}
\sup_{\eps>0}\tilde\rho_\eps(x)>0\qquad\text{for $\Leb{d}$-a.e. $x\in B_s$.}
\end{equation}

Now, for $\tau>s$ and $\eps>0$ fixed, let $B^\tau_s:=\{T_{\O,\sXX}>\tau\}$ and define a generalized flow 
$\eeta_{\tau,\eps}\in\Probabilities{C([s,\tau];\R^d)}$ by
\begin{equation}\label{eqn:semig1}
\eeta_{\tau,\eps}:=\int_{(x,y)\in B^\tau_s\times\{\rho_s\geq\eps\}}\delta_{\sXX(\cdot,x)}\,d\pi_y(x)\,dy=
\int_{B^\tau_s}\delta_{\sXX(\cdot,x)}\,\tilde\rho_\eps(x)\,dx.
\end{equation}
For any $r\in [s,\tau]$ and any $\phi\in C_b(\R^d)$ nonnegative there holds
$$
\int_{\R^d}\phi\,d[(e_r)_\#\eeta_{\tau,\eps}]=
\int_{B^\tau_s}\phi(\XX(r,x))\tilde\rho_\eps(x)\,dx\leq
L\|\tilde\rho_\eps\|_\infty\int_{\R^d}\phi(z)\,dz.
$$
Evaluating at $r=s$, a similar computation gives
$$
(e_s)_\#\eeta_{\tau,\eps}=\XX(s,\cdot)_\#(\chi_{B^\tau_s}\tilde\rho_\eps).
$$
By Theorem~\ref{thm:nosplitting} (applied in the time interval $[s,\tau]$ instead of $[0,T]$)
it follows that 
\begin{equation}\label{eqn:semig22}
\eeta_{\tau,\eps}=\int\delta_{\eta_z}\,d [(e_s)_\#\eeta_{\tau,\eps}](z).
\end{equation}
Now, it is clear that $\WW(\cdot,z):=\eta_z(\cdot)$ is a regular flow in $[s,\tau]$, hence (by uniqueness)
$\eta_z=\XX(\cdot,s,z)$ for $(e_s)_\#\eeta_{\tau,\eps}$-a.e. $z$. Returning to \eqref{eqn:semig22} we get
\begin{equation}\label{eqn:semig2}
\eeta_{\tau,\eps}=\int\delta_{\sXX(\cdot,s,z)}\,d [(e_s)_\#\eeta_{\tau,\eps}](z)=
\int_{B^\tau_s}\delta_{\sXX(\cdot,s,\sXX(s,x))}\tilde\rho_\eps(x)\,dx,
\end{equation}
where in the second equality we used the formula for $(e_s)_\#\eeta_{\tau,\eps}$.
Comparing formulas \eqref{eqn:semig1} and \eqref{eqn:semig2}, and taking
\eqref{eqn:nohhole} into account, we find
that $T_{\O,\sXX,s}(\XX(s,x))\geq \tau-s$ and that
$\XX\bigl(\cdot,s,\XX(s,x)\bigr)\equiv \XX(\cdot+s,x)$ in $[s,\tau]$, for $\Leb{d}$-a.e. $x\in B^\tau_s$. Since $\tau>s$ is
arbitrary, it follows that $T_{\O,\sXX,s}(\XX(s,x))\geq T_{\O, \sXX}(x)-s$ and that
$\XX\bigl(t,s,\XX(s,x)\bigr)=\XX(t+s,x)$ $\Leb{d}$-a.e. in $B_s$.

If $T_{\O,\sXX}(x)<T$, by the semigroup identity it follows that
$$\limsup\limits_{t\uparrow T_{\O,\ssXX}(x)-s}V_\O(\XX\bigl(t,s,\XX(s,x)\bigr))=\limsup\limits_{t\uparrow T_{\O,\ssXX}(x)-s}V_\O(\XX(t+s,x)) =\infty,$$
and hence
\begin{equation}
\label{eqn:tsox-ineq} T_{\O,\sXX,s}(\XX(s,x)) = T_{\O, \sXX}(x)-s
\qquad\text{for $\Leb{d}$-a.e. $x\in B_s$.}
\end{equation} 

Eventually we use the arbitrariness of $B_s$ to conclude \eqref{eqn:consecutive-o} and \eqref{eqn:semig3-o-int}.
\end{Proof}

\subsection{Stability}

The following theorem provides a stability result for maximal regular flows in $\O$ when the vector fields converge strongly in space and weakly in time, in 
analogy with the classical theory (see also Remark~\ref{rem:weakspacetime} below).

\begin{theorem}[Stability of maximal regular flows in $\O$]\label{thm:stab-o}
Let $\O\subset\R^d$ be an open set. 
Let $\XX^n$ be maximal regular flows in $\O$ relative to locally integrable Borel vector fields 
$\bb^n:(0,T)\times\O\to\R^d$.
Assume that:
\begin{itemize}
\item[(a)] for any $A\Subset\O$ open the compressibility constants $C(A,\XX^n)$ in Definition~\ref{def:maxflow-o} are uniformly bounded;
\item[(b)] 
for any $A\Subset\O$ open, setting $A^\eps:= \{ x\in A: {\rm dist}(x,\R^d \setminus A) \geq \eps\}$ for $\eps>0$, there holds
\begin{equation}\label{eqn:conv-b-forte-deb}
\lim_{h\to 0}\chi_{A^{|h|}}(x+h)\bb^n(t,x+h)=\chi_{A}(x)\bb^n(t,x)\quad\text{in $L^1((0,T)\times A)$, uniformly w.r.t. $n$;}
\end{equation}
\item[(c)] there exists a Borel vector field $\bb:(0,T)\times\O\to\R^d$ satisfying (a-$\O$) and (\bbb-$\O$) such that
\begin{equation}
\label{hp:conv-b}
\bb^n \rightharpoonup \bb
\qquad \mbox{weakly in } L^1((0,T) \times A)\qquad\text{for all $A\Subset\O$ open.}
\end{equation}
\end{itemize}
Then there exists a unique maximal regular flow $\XX$ for $\bb$ and, for every $t\in [0,T]$ and any open set $A\Subset \O$, we have 
\begin{equation}
\label{eqn:stab-o}
\lim_{n\to \infty}\Bigl\| \max_{s\in [0,t]} |\XX^n_A(s,\cdot) - \XX(s,\cdot)|\wedge 1 \Bigr\|_{L^1(\{ x:\ \hit{A}{\sXX(\cdot,x)} >t \})} =0,
\end{equation}
where
$$
\XX^n_A(t,x):=
\left\{
\begin{array}{ll}
\XX^n(t,x) &
\text{for $t\in [0,\hit{A}{\XX^n(\cdot,x)}]$,}\\
\XX^n(\hit{A}{\XX^n(\cdot,x)),x} & \text{for $t\in [\hit{A}{\XX^n(\cdot,x)},T]$}.
\end{array}
\right.
$$
\end{theorem}

\begin{remark}\label{rem:weakspacetime}{\rm
The convergence \eqref{hp:conv-b} and \eqref{eqn:conv-b-forte-deb} of $\bb^n$ to $\bb$ is implied by the strong convergence of $\bb^n$ to $\bb$ in space-time. It is however quite natural to state the convergence in these terms in view of some applications. For example, the weak convergence of \eqref{hp:conv-b} and the boundedness in a fractional Sobolev space $\bb^n \in L^1((0,T); W^{m,p}(\R^d))$,
 $p > 1, m > 0 $,  is enough to guarantee that \eqref{eqn:conv-b-forte-deb} holds.  The same kind of convergence appears in \cite[Theorem II.7]{lions} to prove convergence of distributional solutions of the continuity equation, and in \cite[Remark 2.11]{crippade} in the context of quantitative estimates on the flows of Sobolev vector fields.

The convergence of the vector fields in \eqref{eqn:stab-o} is localised to the trajectories of $\bb$ which are inside $A$ in $[0,t]$. This is indeed natural: even with smooth vector fields one can construct examples where the existence time of $\XX(\cdot, x)$ is strictly smaller than the existence time of $\XX^n(\cdot, x)$ and the convergence of $\XX^n(\cdot, x)$ to $\XX(\cdot, x)$, or to its constant extension beyond the existence time $T_{\O, \sXX}(x)$, fails after $T_{\O, \sXX}(x)$ (see Figure~\ref{fig:convergence}).
\begin{figure}[]
\center \includegraphics[scale=0.60]{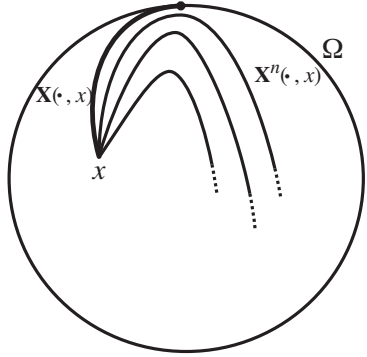}
\caption{One can build a sequence of smooth vector fields $\bb^n$ whose trajectories $\XX^n(\cdot, x)$ starting from a point $x$ is drawn in the figure. These trajectories fail to converge to the constant extension of $\XX(\cdot, x)$ after $T_{\O,\sXX}(x)$.}
\label{fig:convergence}
\end{figure} 
}\end{remark}

The stability of maximal flows in Theorem~\ref{thm:stab-o} implies a lower semicontinuity property of hitting times.
\begin{corollary}[Semicontinuity of hitting times]
With the same notation and assumptions of Theorem~\ref{thm:stab-o}, for every $t\in [0,T]$ we have that
\begin{equation}
\label{eqn:stab-hitting}
\lim_{n\to \infty}\Leb{d}\big( \{ x:\ \hit{A}{\XX^{n}(\cdot,x)} \leq t< \hit{A}{\XX(\cdot,x)} \} \big)=0.
\end{equation}¥In particular, there exists a subsequence $n(k) \to \infty$ (which depends, in particular, on $A$) such that
\begin{equation}
\label{eqn:hitting-lsc}
\hit{A}{\XX(\cdot,x)} \leq
\liminf_{k\to\infty}\hit{A}{\XX^{n(k)}(\cdot,x)}
\quad\text{$\Leb{d}$-a.e. in $A$.}
\end{equation}¥\end{corollary}
\begin{proof}
For every $x$ such that $\hit{A}{\XX^{n}(\cdot,x)} \leq t< \hit{A}{\XX(\cdot,x)}$ we have that 
$$\max_{s\in [0,t]} |\XX^n_A(s,x) - \XX(s,x)| \geq {\rm dist} (\partial A, \XX([0,t],x))>0.$$
It implies, together with \eqref{eqn:stab-o}, that \eqref{eqn:stab-hitting} holds.

Up to a subsequence and with a diagonal argument, by \eqref{eqn:stab-hitting} we deduce that for every $t\in \Q \cap [0,T]$ the functions $1_{\{\hit{A}{\sXX^{n(k)}(\cdot,x)} \leq t\}}$ converge pointwise a.e. to $0$ in $\{\hit{A}{\XX(\cdot,x)} > t\}$ and therefore for a.e. $x$ such that $t< \hit{A}{\XX(\cdot,x)}$ we have $\hit{A}{\XX^{n}(\cdot,x)} >t$ for $n$ large enough.
This implies that for every $t\in \Q \cap [0,T]$, for a.e. $x$ such that $t< \hit{A}{\XX(\cdot,x)}$ we have
$$ t \leq
\liminf_{k\to\infty}\hit{A}{\XX^{n(k)}(\cdot,x)}
\quad\text{$\Leb{d}$-a.e. in $A$,}
$$
which implies \eqref{eqn:hitting-lsc}.
\end{proof}

The proof of the stability of maximal regular flows in $\O$ is based on a tightness and stability result for regular generalized flows in $\overline{A}$
(according to Definition~\ref{def:dregbar}), as the one presented in Theorem~\ref{thm:tightness-o} under the assumption of the strong space-time convergence of the vector fields.

\begin{proposition}[Tightness and stability of generalized regular flows]\label{thm:tightness-o-fortedeb}
Let $A \subset \R^d$ be a bounded open set.
The result of Theorem~\ref{thm:tightness-o} holds true also if we replace the strong convergence of the vector fields \eqref{eqn:bn-conv} with the assumptions 
\begin{equation}\label{eqn:conv-b-forte-deb-bis}
\lim_{h\to 0}\chi_{A^{|h|}}(x+h)\cc^n(t,x+h)=\chi_A(x)\cc^n(t,x)\quad\text{in $L^1((0,T)\times A)$, uniformly w.r.t. $n$,}
\end{equation}
\begin{equation}\label{hp:conv-b-bis}
\cc^n \rightharpoonup \cc \qquad \mbox{weakly in } L^1((0,T) \times A),
\end{equation}
where $A^\eps:= \{ x\in A: {\rm dist}(x,\R^d \setminus A) \geq \eps\}$ for $\eps>0$ (compare with \eqref{eqn:conv-b-forte-deb} and \eqref{hp:conv-b}).
\end{proposition}
\begin{Proof}
The tightness was based on Dunford-Pettis' theorem and it can be repeated in this context
thanks to \eqref{hp:conv-b-bis}:
in particular,
there exists a modulus of integrability $F$ such that
\begin{equation}
\label{eqn:mod-of-int}
\sup_{n\in \N} \int\int_0^TF(|\dot\eta(t)|)\,dt\,d\eeta^n < \infty.
\end{equation}
We show that $\eeta$ is concentrated on integral curves of $\cc$, namely
\begin{equation}\label{eqn:integral-curves}
\int\left|\eta(t)-\eta(0)-\int_0^t\cc(s,\eta(s))\,ds\right|\,d\eeta(\eta)=0
\end{equation}
for any $t\in [0,T]$.
To this end we consider $\cc^\eps := (\cc \chi_{A^\eps}) \ast \rho_\eps$, where $\rho_\eps(x):= \eps^n \rho(x/\eps)$, $\rho \in C^\infty_c(\R^d)$ nonnegative, 
is a standard convolution kernel in the space variable with compact support in the unit ball.
Notice that $\cc^\eps \in L^1((0,T); C^\infty_c(A;\R^d))$ and that $|\cc^\eps-\cc|\to 0$ in $L^1((0,T) \times A)$ as $\eps \to 0$. 
Similarly, for every $n\in \N$ we set $\cc^{n,\eps} := (\cc^n \chi_{A^\eps}) \ast \rho_\eps$.
We first prove that, for every $\eps>0$, 
\begin{equation}\label{eqn:integral-curves-eps}
\int\left|\eta(t)-\eta(0)-\int_0^t\cc^\eps(s,\eta(s))\,ds\right|\,d\eeta(\eta) \leq \omega(\eps),
\end{equation}
where $\omega:(0,\infty) \to (0,\infty)$ is a nondecreasing function which goes to $0$ as $\eps \to 0$ to be chosen later.

Since the integrand is a continuous (possibly unbounded) function of $\eta \in C([0,T]; \R^d)$ and $\eeta^n$ is concentrated 
on integral curves of $\cc^n$, by the triangular inequality we have the estimate
\begin{eqnarray}
\label{eqn:integral-curves-eps1}
&&\int\left|\eta(t)-\eta(0)-\int_0^t\cc^\eps(s,\eta(s))\,ds\right|\,d\eeta(\eta)\\\nonumber&\leq&
\liminf_{n\to\infty} \int\left|\eta(t)-\eta(0)-\int_0^t\cc^\eps(s,\eta(s))\,ds\right|\,d\eeta^n(\eta) 
\\\nonumber
&\leq&
\liminf_{n\to\infty} \left[ \int\left|\int_0^t[\cc^{n}-\cc^{n,\eps}](s,\eta(s))\,ds\right|\,d\eeta^n(\eta) 
+
\int\left|\int_0^t[\cc^{n,\eps}-\cc^\eps](s,\eta(s))\,ds\right|\,d\eeta^n(\eta).
 \right]
\end{eqnarray}
To estimate the first term in the right-hand side of \eqref{eqn:integral-curves-eps1}, we notice that 
$$\sup_{n\in \N} \|\cc^{n, \eps} - \cc^n\|_{L^1((0,T) \times A)} \leq \omega(\eps)$$
and $\omega(\eps)  \to 0$ as $\eps \to 0$.
Indeed, consider a nondecreasing function $\omega_0:(0,\infty) \to (0,\infty)$ which goes to $0$ as $\eps \to 0$ and such that
 \begin{equation}
\label{defn:omega}
\|\chi_{A^{|h|}}(x-h)\cc^n(t,x-h)-\chi_A(x)\cc^n(t,x)\|_{L^1((0,T)\times A)} \leq \omega_0(|h|)
\end{equation}
for every $n \in \N$, which exists thanks to \eqref{eqn:conv-b-forte-deb-bis}.
We notice that
\begin{equation*}
\begin{split}
\int_0^T \int_A |\cc^{n, \eps} - \cc^n| \, dx\, dt 
&\leq
\int_{\R^d} \rho_\eps(z) \int_0^T \int_A |\chi_{A^{\eps}}(x-z)\cc^{n}(t,x-z) - \cc^n(t,x)| \, dx\, dt \, dz
\\
&\leq
\int_{\R^d} \rho_\eps(z) \int_0^T \int_A [\chi_{A^{|z|}}(x-z)-\chi_{A^{\eps}}(x-z)] |\cc^{n}(t,x-z)|\, dx\, dt \, dz
\\
&\hspace{1em} +\int_{\R^d} \rho_\eps(z) \int_0^T \int_A |\chi_{A^{|z|}}(x-z)\cc^{n}(t,x-z) - \cc^n(t,x)| \, dx\, dt \, dz
\\
&\leq
\int_{\R^d} \rho_\eps(z) \int_0^T \int_{\R^d} [\chi_{A}(x)-\chi_{A^{\eps}}(x)] |\cc^{n}(t,x)|\, dx\, dt \, dz +\omega_0(\eps)
\end{split}
\end{equation*}
and the first term converges to $0$ uniformly in $n$ thanks to \eqref{hp:conv-b-bis}, Dunford-Pettis' theorem and since $A^\eps \uparrow A$ as $\eps \to 0$.

Hence, using the fact that $\cc^n=0$ on $\partial A$ and the definition \eqref{eqn:noconcentration-o} of compressibility constant $C_n$ for
$\eeta^n$ we get
\begin{equation}
\label{eqn:integral-curves-eps2} \int\left|\int_0^t[\cc^{n}-\cc^{n,\eps}](s,\eta(s))\,ds\right|\,d\eeta^n(\eta) 
\leq
C_n \int_\O\int_0^t|\cc^{n}-\cc^{n,\eps}|\,ds\, dx 
\leq \sup_n C_n \,\omega(\eps).
\end{equation}
We now estimate the second term in the right-hand side of \eqref{eqn:integral-curves-eps1}. To this end, for every $k>0$ we consider the set of curves
$$\Gamma_k := \Big\{ \eta \in AC([0,T]; \overline{A}): \int_0^TF(|\dot\eta(t)|)\,dt \leq k \Big\}.$$
We notice that all curves in $\Gamma_k$ have a uniform modulus of continuity that we 
denote by $\tilde \omega_k$. 
By Chebyshev's inequality and \eqref{eqn:mod-of-int} we deduce that
$$\eeta^n(C([0,T]; \overline{A}) \setminus \Gamma_k) \leq \frac{C}{k}$$
for some constant $C>0$, hence in
 the complement of $\Gamma_k$ we estimate the integrand with its $L^\infty$ norm:
\begin{equation}
\label{eqn:integral-curves-eps3}
\begin{split}
\int_{\Gamma^c_k}
\left|\int_0^t[\cc^{n,\eps}-\cc^\eps](s,\eta(s))\,ds\right|\,d\eeta^n(\eta) 
&\leq
\eeta^n( \Gamma^c_k) \int_0^T \| [\cc^{n,\eps}-\cc^\eps](s,\cdot)\|_{L^\infty(A)} \, ds
\\
&\leq
\frac{C}{k} \|\cc^n-\cc\|_{L^1((0,T) \times A)} \|\rho_\eps\|_{L^\infty(A)}.
\end{split}
\end{equation}
Hence, choosing $k$ large enough we can make this term as small as we wish uniformly with respect to $n$, since $\|\cc^n-\cc\|_{L^1((0,T) \times A)} \leq \|\cc^n\|_{L^1((0,T) \times A)}+ \|\cc\|_{L^1((0,T) \times A)}$ is bounded.

In $\Gamma_k$, for any $N\in\N$ we can use the triangular inequality, the fact that $\cc^{n,\eps}$ and $\cc^\eps$ are null on $(0,T)\times\partial A$, and 
the bounded compression condition $(e_{i/N})_\# \eeta^n \res A \leq C_n \Leb{d}$ for every $i=1,\ldots,N$, to get
\begin{eqnarray}
\label{eqn:integral-curves-eps4}
\int_{\Gamma_k}\left|\int_0^t[\cc^{n,\eps}-\cc^\eps](s,\eta(s))\,ds\right|\,d\eeta^n(\eta) 
&\leq&
\sum_{i=1}^N \int_{\Gamma_k}\left|\int_{t^N_{i-1}}^{t^N_i}[\cc^{n,\eps}-\cc^\eps](s,\eta(s))\,ds\right|\,d\eeta^n(\eta) 
\\
&\leq &\nonumber
\sum_{i=1}^N \int_{\Gamma_k}\left|\int_{t^N_{i-1}}^{t^N_i}[\cc^{n,\eps}-\cc^\eps]\big(s,\eta\big(t^N_i\big)\big)\,ds\right|\,d\eeta^n(\eta) 
\\ \nonumber
&&+\tilde \omega_k\Big( \frac{t}{N}\Big)
\sum_{i=1}^N \int_{t^N_{i-1}}^{t^N_i}\|\nabla[\cc^{n,\eps}-\cc^\eps](s,\cdot)\|_{L^\infty(A)}\,ds 
\\
&\leq&\nonumber
C_n\sum_{i=1}^N \int_A\left|\int_{t^N_{i-1}}^{t^N_i}[\cc^{n,\eps}-\cc^\eps]\,ds\right|\,dx
\\
&&\nonumber +
\tilde \omega_k\Big( \frac{t}{N}\Big)
\|\cc^n-\cc\|_{L^1((0,T) \times A)} \|\nabla \rho_\eps\|_{L^\infty(\R^d)},
\end{eqnarray}
where $t^N_i = it/N$.
Choosing $N$ large enough we can make the second term in the right-hand side as small as we want, uniformly in $n$.
Letting $n \to \infty$ in \eqref{eqn:integral-curves-eps4}, each term in the first sum in the right-hand side converges to $0$ pointwise in $x$ by the weak convergence \eqref{hp:conv-b} tested with the function $\varphi_x(s,y) = 1_{[t^N_{i-1}, t^N_i]}(s) \rho_\eps(x-y)$, namely, for every $x\in A$,
$$
\lim_{n \to \infty} \int_{t^N_{i-1}}^{t^N_i}[\cc^{n,\eps}-\cc^\eps](s,x)\,ds
=
\lim_{n \to \infty} \int_{t^N_{i-1}}^{t^N_i}[\cc^n-\cc](s,y) \rho_\eps(x-y)\,ds
=0.
$$
These functions are bounded by $\|\cc^n-\cc\|_{L^1((0,T) \times A)} \|\rho_\eps\|_{L^\infty(\R^d)}$, thus by dominated convergence the first sum in the right-hand side of \eqref{eqn:integral-curves-eps4} converges to $0$. It follows that, given $\eps$ and $k$, by choosing $N$ sufficiently large we can make also this term
as small as we wish, hence \eqref{eqn:integral-curves-eps}
follows from \eqref{eqn:integral-curves-eps1}.
We now let $\eps \to 0$ in \eqref{eqn:integral-curves-eps} 
and  notice that, since $\eeta$ satisfies \eqref{eqn:noconcentration-o} with $C=\liminf_n C_n$
and $\cc^\eps \to \cc$ in $L^1((0,T) \times A)$,
$$\lim_{\eps \to 0} \int\left|\int_0^t[\cc-\cc^{\eps}](s,\eta(s))\,ds\right|\,d\eeta(\eta) 
\leq C\lim_{\eps \to 0} \int_A\int_0^t|\cc-\cc^{\eps}|\,ds\, dx =0,
$$
proving the validity of \eqref{eqn:integral-curves}.
\end{Proof}

The following lemma is a standard tool in optimal transport theory (for a proof, see for instance \cite[Lemma~22]{cetraro}, or \cite[Corollary 5.23]{villani09}).

\begin{lemma}\label{lemma:misure} Let $X_1,\,X_2$ be Polish metric spaces, let $\mu \in \Probabilities{X_1}$, and let 
$F_n : X_1 \to X_2$ be a sequence of Borel functions. If 
\begin{equation}
\label{eqn:plans-conv}
({\rm Id},F_n)_\# \mu \rightharpoonup ({\rm Id},F)_\# \mu\qquad\text{narrowly in $\Probabilities{X_1\times X_2}$},
\end{equation}
then $F_n$ converge to $F$ in $\mu$-measure, namely
$$
\lim_{n\to \infty} \mu( \{ d_{X_2}(F_n,F) >\eps\} ) =0 \qquad \forall \,\eps>0.
$$
\end{lemma}

\begin{proof}[Proof of Theorem~\ref{thm:stab-o}]
Fix $A\Subset\O$ open, denote by $\Leb{d}_A$ the normalized Lebesgue measure on $A$,
and define  $\XX^n_A$ as in the statement of the theorem. Then the laws
$\eeta^n$ of $x\mapsto\XX^n_A(\cdot,x)$ under $\Leb{d}_A$ define regular generalized flows in $\overline{A}$ relative to
$\cc^n=\chi_A\bb^n$, according to
Definition~\ref{def:dregbar}, with compressibility
constants $C_n= C(A,\XX^n)$.  

Hence we can apply Proposition~\ref{thm:tightness-o-fortedeb} to obtain that, up to a subsequence, $\eeta^n$ weakly converge to a generalized  flow 
$\eeta$ in $\overline{A}$ relative to the vector field $\cc=\chi_A\bb$, with compressibility constant $C=\liminf_n C_n$. Let $\eeta_x$
be the conditional probability measures induced by the map $e_0$, and let
$\XX_A$ and $T_A$ be given by Proposition~\ref{prop:auxiliary}; recall that $\XX_A(\cdot,x)$
is an integral curve of $\bb$ in $[0,T_A(x)]$, that $\XX_A([0,T_A(x)),x)\subset A$, and that  $\XX_A(T_A(x),x)\in\partial A$ if $T_A(x)<T$; 
as explained in Remark~\ref{rem:auxiliary}, 
for $\Leb{d}_A$-almost every $x$ the hitting time $\hit{A}{\eta}$ is equal to $T_A(x)$ for $\eeta_x$-a.e. $\eta$, and 
$(e_t)_\#\eeta_x =\delta_{\sXX_A(t,x)} $ for all $t\in [0,T_A(x)]$.
For every $t\in [0,T]$ we set $E_{t,A} := \{ T_A(x) >t\}$;  since
$$\XX_A(s, \cdot)_\#( \Leb{d} \res E_{t,A}) = (e_s)_\# \int_{E_{t,A}} \delta_{\sXX_A(\cdot, x)} \, d\Leb{d} \leq (e_s)_\# \eeta \leq  C \Leb{d} \qquad \forall\,s\in [0,t], $$ 
we obtain that $\XX_A$ is a regular flow for $\bb$ on $[0,t] \times E_t$. Applying Theorem~\ref{thm:main-o}(b) to $\XX_{A_1}$ and $\XX_{A_2}$ with $A_1\subset A_2$
we deduce that $\XX_{A_1} = \XX_{A_2}$ on $E_{t,A_1}$, and this allows us (by a gluing procedure)
to obtain a maximal regular flow for $\bb$.

To prove the last statement, 
we apply Lemma~\ref{lemma:misure} with $X_1 = \R^d$, $\mu = (\Leb{d} \res \{ T_A >t \}) / \Leb{d} ( \{ T_A >t \} )$, $X_2 = C([0,t]; \overline{A})$, 
$F_n (x)= \XX^n_A(\cdot, x)$, $F(x)=\XX_A(\cdot,x)$. 
More precisely, we consider the laws $\tilde \eeta^n\in\Probabilities{C([0,t];\R^d)}$ of $x\mapsto\XX^n_A(\cdot,x)$ under $\mu$; with the same argument as above, we know that $\tilde \eeta^n$ weakly converge to $\tilde \eeta$ and that the disintegration $\tilde \eeta_x$ coincides with $\delta_{\sXX_A(\cdot,x)}$ for $\mu$-a.e. $x\in \R^d$ (notice that $\XX_A(\cdot,x)$ is defined in $[0,t]$ for $\mu$-a.e. $x$).
The assumption \eqref{eqn:plans-conv} is satisfied, since for every bounded continuous function 
$\varphi: \R^d \times  C([0,T]; \overline{A}) \to \R$ we have
$$
\int \varphi(x,\gamma) \, d ({\rm Id},  \XX^n_A(\cdot, x))_\# \mu(x,\gamma)
=
\int \varphi(\gamma(0), \gamma ) \, d \tilde \eeta^n(\gamma)
$$
(and similarly with $\tilde \eeta$) and the weak convergence of $\tilde \eeta^n$ to $\tilde \eeta$ shows that
$$
\lim_{n\to \infty} \int \varphi(x,\gamma) \, d ({\rm Id},  \XX_A^n(\cdot, x))_\# \mu(x, \gamma) = 
\int \varphi(x,\gamma) \, d ({\rm Id},  \XX_A(\cdot, x))_\# \mu(x, \gamma).
$$
We deduce the convergence in $\mu$-measure of $\XX_A^n$ to $\XX_A$ in $C([0,t];\overline{A})$, i.e.,
$$\lim_{n\to \infty} \Leb{d}\biggl(
\Bigl\{x \in \{T_A>t\}:\ \sup_{s\in [0,t]} |\XX_A^n(s,x)-\XX_A(s,x)| >\eps\Bigr\}\biggr) = 0
\qquad\forall\,\eps>0,
$$
from which \eqref{eqn:stab-o} follows easily.
\end{proof}

\section{Further properties implied by global bounds on divergence}

\subsection{Proper blow up of trajectories}\label{sec:proper-blow-up}

Recall that the blow-up time $T_{\sox}(x)$ for maximal regular flows is characterized by the property
$\limsup_{t\uparrow T_{\sox}(x)}V_\O(\XX(t,x))=\infty$ when $T_{\sox}(x)<T$. We say that $\XX(\cdot,x)$ blows up {\it properly}
(i.e. with no oscillations) if the stronger condition $\lim_{t\uparrow T_{\sox}(x)}V_\O(\XX(t,x))=\infty$ holds. In the following
theorem we prove this property when a global bounded compression condition on $\XX$ is available, see 
\eqref{eq:pasquetta1} below. Thanks to the properties of the maximal regular flow the global bounded compression
condition is fulfilled, for instance, in all
cases when the divergence bounds $L(\O')$ in \eqref{eqn:incompb-o-L} are uniformly bounded. More precisely 
\begin{equation}\label{eqn:incompb-o-g}
{\rm div\,}\bb(t,\cdot)\geq m(t)\quad\text{in $\O$, with $L(\O):=\int_0^T |m(t)|\, dt<\infty$}
\end{equation}
implies \eqref{eq:pasquetta1} with $C_*\leq e^{L(\O)}$.

\begin{theorem}\label{thm:properblo}
Let $\XX$ be a maximal regular flow relative to a Borel vector field $\bb$ satisfying 
(a-$\O$) and (\bbb-$\O$), and assume that the bounded compression 
condition is global, namely there exists a constant $C_*\geq 0$ satisfying
\begin{equation}\label{eq:pasquetta1}
\XX(t,\cdot)_\#\Leb{d}\res \{T_{\sox}>t\}\leq C_*\Leb{d}\qquad\forall \,t\in [0,T).
\end{equation}
Then 
$$\liminf_{t\uparrow T_{\sXX}(x)}|\XX(t,x)|= \infty \qquad\text{for $\Leb{2}$-a.e. $x\in\R^2$ such that }\limsup_{t\uparrow T_{\sXX}(x)}|\XX(t,x)|=\infty,$$
and in particular $\lim_{t\uparrow T_{\sox}(x)}V_\O(\XX(t,x))=\infty$ for $\Leb{d}$-a.e. $x$ with $T_{\sox}(x)<T$.
\end{theorem}
\begin{Proof} Let $\O_n$ be open sets with $\O_n\Subset\O_{n+1}\Subset\O$, with
$\cup_n\O_n=\O$. We consider cut-off functions $\psi_n\in C^\infty_c(\O_{n+1})$ with $0\leq\psi_n\leq 1$
and $\psi_n\equiv 1$ on a neighborhood of $\overline{\O}_n$.

Since $\XX(\cdot,x)$ is an integral curve of $\bb$ for $\Leb{d}$-a.e $x\in\O$ we can use
\eqref{eq:pasquetta1} to estimate
\begin{equation}
\begin{split}
\int_{\O} \int_0^{T_{\sox}(x)} \Big| \frac{d}{dt} \psi_n(\XX(t,x))\Big| \, dt\, dx 
&\leq
\int_{\O} \int_0^{T_{\sox}(x)}| \nabla \psi_n(\XX(t,x))| \, |\bb(t,\XX(t,x))| \, dt\, dx 
\\&=
\int_0^T \int_{\{T_{\sox}>t\}} |\nabla \psi_n(\XX(t,x))|\, |\bb(t,\XX(t,x))| \, dx \, dt
\\
&\leq 
C_*\int_0^T \int_{\R^d} |\nabla \psi_n(y)| |\bb(t,y)| \, dy \, dt
\\
&\leq 
C \|\nabla \psi_n \|_{L^\infty(\O)} \int_0^T \int_{\O_{n+1}} |\bb(t,x)| \, dx \, dt.
\end{split}
\end{equation}
Hence $\psi_n(\XX(\cdot, x))$ is the restriction of an absolutely continuous map in $[0,T_{\sox}(x)]$ 
(and therefore uniformly continuous in $[0,T_{\sox}(x))$) for $\Leb{d}$-a.e. $x\in\O$.

Let us fix $x\in\O$ such that $\limsup_{t\uparrow T_{\sox}(x)}V_\O(\XX(t,x))=\infty$
 and $\psi_n(\XX(\cdot, x))$ is uniformly continuous in $[0,T_{\sox}(x))$ for every $n\in \N$. 
The $\limsup$ condition yields that the limit of all $\psi_n(\XX(t,x))$ as $t\uparrow T_{\sox}(x)$
must be 0. On the other hand, if the $\liminf$ of $V_\O(\XX(t,x))$ as $t\uparrow T_{\sox}(x)$ were finite,
we could find an integer $n$ and $t_k\uparrow T_{\sox}(x)$ with $\XX(t_k,x)\in\O_n$ for all $k$. Since $\psi_{n+1}(\XX(t_k,x))=1$ we obtain a contradiction.
\end{Proof}

\begin{remark}\label{rmk:flow-distr-sol}
{\rm
Under the assumptions of the previous theorem, given any probability measure $\mu_0 \leq C\Leb{d}$ for some $C>0$, it can be easily shown that the measure
 \begin{equation}\label{eqn:sol-eq-cont-flow}
 \mu_t := \XX(t,\cdot)_\#(\mu_0\res\{T_{\sXX}>t\}), \qquad t\in[0,T]
 \end{equation}
 is a bounded (by Theorem~\ref{thm:properblo}), weakly* continuous, distributional solution to the continuity equation.
We notice that the same statement is not true if we assume only a local bound on $\div \bb$, since the measure \eqref{eqn:sol-eq-cont-flow} can be locally unbounded, as in the example of Proposition~\ref{prop:no-prop-blowup}, and therefore we cannot write the distributional formulation of the continuity equation.

To see that \eqref{eqn:sol-eq-cont-flow} is a distributional solution of the continuity equation, we consider $\varphi \in C^\infty_c(\R^d)$ and we define the function $g(t,x)$  as $\varphi(\XX(t,x))$ if $t< T_X(x)$ or $t= T_X(x)=T$, and $g(t,x) = 0$ otherwise.
We notice that $g(t,x)$ is absolutely continuous with respect to $t$  for a.e. $x \in \R^d$ and that $\frac{d}{dt}g(t,x) = 1_{\{ T_{\sXX}(x) >t\}}  \nabla \varphi(\XX(t,x)) \bb(t,\XX(t,x))$
for a.e. $t \in (0,T)$, for a.e. $x\in \R^d$. We deduce that the function $t \to \int_{\{T_\sXX>t\}}\varphi(\XX(t,x))\,dx$ is absolutely continuous and its derivative is given by
\begin{equation*}
\begin{split}
\frac{d}{dt} \int_{\{T_\sXX>t\}}\varphi(\XX(t,x))\,d\mu_0(x)
&= 
\frac{d}{dt} \int_{\R^d}g(t,x)\,d\mu_0(x)
\\&
= \int_{\{T_\sXX>t\}}\nabla \varphi(\XX(t,x)) \bb(t,\XX(t,x))\,d\mu_0(x).
\end{split}
\end{equation*}
}\end{remark}

The proper blow up may fail for the maximal regular flow due only to the lack of a global bound on the divergence of $\bb$, as shown in the next example. 

In the following we denote by $\ee_1,\ldots, \ee_d$ the canonical basis of $\R^d$ and  
$B^{(d-1)}_{r} (x')\subset \R^{d-1}$ the ball of center $x' \in \R^{d-1}$ and radius $r$. We denote each point $x\in \R^d$ as $x = (x', x_n)$, where $x'$ are the first $d-1$ coordinates of $x$. For simplicity we write $T_{\sXX}$ for $T_{\R^d,\sXX}$.

\begin{proposition}\label{prop:no-prop-blowup}
Let $d\geq 3 $. There exist an autonomous vector field $\bb:\R^d \to \R^d$ and a Borel set of positive measure $\Sigma\subset \R^d$ such that
$\bb \in W^{1,p}_{\rm loc} (\R^d;\R^d)$ for some $p>1$, $\div \bb \in L^\infty_{\rm loc} (\R^d)$,  
and
\begin{equation}\label{eq:blowup-c}
T_{\sXX}(x)\leq 2, \qquad \liminf_{t\uparrow T_{\sXX}(x)}|\XX(t,x)|= 0, \qquad \limsup_{t\uparrow T_{\sXX}(x)}|\XX(t,x)|=\infty
\end{equation} 
for every $x\in\Sigma$.
\end{proposition}
\begin{Proof}
\begin{figure}[]
\centering
\includegraphics[scale=0.40]{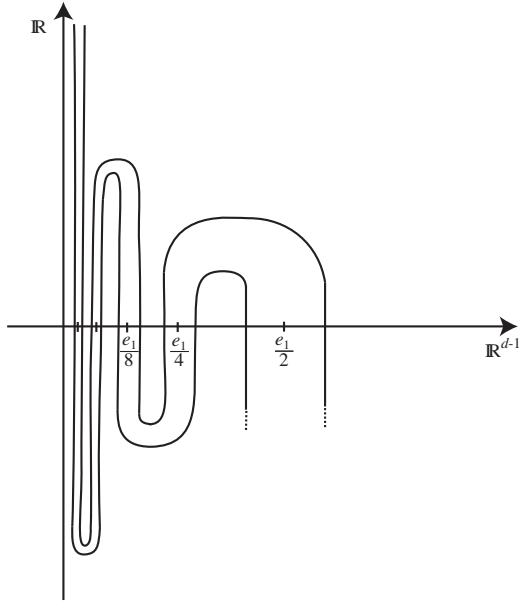}
\caption{The trajectories of $\bb$ oscillate between $0$ and $\infty$.}
\label{fig:traiett-campo}
\end{figure} 
We build a vector field whose trajectories are represented in Figure~\ref{fig:traiett-campo}.

Let $\{a_k\}_{k\in \N}$ be a fastly decaying sequence to be chosen later. For every $k=1,2,\ldots$ we define the cylinders 
\begin{equation*}
E_k  =
\begin{cases}
B^{(d-1)}_{a_k} ( 2^{-k}\ee_1) \times [-2^{k-1},2^{k}]
 \qquad&\text{if $k$ is odd}\\
B^{(d-1)}_{a_k} ( 2^{-k}\ee_1) \times [-2^k,2^{k-1}]
\qquad&\text{if $k$ is even.}
\end{cases}
\end{equation*}
We also define
$$E_0= B^{(d-1)}_{a_1} ( 2^{-1}\ee_1) \times (-\infty,-1].$$
Let $\varphi \in C^\infty_0 (B^{(d-1)}_1)$ be a nonnegative cutoff function which is equal to $1$ in $B_{1/2}$. In every $E_k$ the vector field $\bb$ points in the $d$-th direction and it depends only on the first $d-1$ variables
\begin{equation}
\label{defn:b-cilindri}
\bb (x) :=
\begin{cases}
\displaystyle{(-1)^{k+1} 4^k \varphi\Big( \frac{x'-2^{-k}\ee_1}{a_k} \Big) \ee_d }\qquad &\forall \,x\in E_k,\,\,k\geq 1
\\
\displaystyle{4 \varphi\Big( \frac{x'-2^{-1}\ee_1}{a_1} \Big) \ee_d} \qquad &\forall \,x\in E_0.
\end{cases}
\end{equation}
Notice that $\div \bb =0$ in every $E_k$ and that $\bb$ is $0$ on the lateral boundary of every cylinder $E_k$ since $\varphi$ is compactly supported.

For every $k\geq 1$ we define the cylinders $E_k' \subset \R^d$ as
\begin{equation*}
E_k'  =
\begin{cases}
B^{(d-1)}_{a_k/2} ( 2^{-k}\ee_1) \times [-2^{k-1},2^{k}]
 \qquad&\text{if $k$ is odd}\\
B^{(d-1)}_{a_k/2} ( 2^{-k}\ee_1) \times [-2^k,2^{k-1}]
\qquad&\text{if $k$ is even.}
\end{cases}
\end{equation*}
For every $k \in \N$ we define a handle $F_k$ which connects $E_k$ with $E_{k+1}$ as in Figure~\ref{fig:campo-EF}. It is made of a family of smooth, nonintersecting curves of length less than $1$ which connect the top of $E_k$ to the top of $E_{k+1}$ and $E_k'$ with $E_{k+1}'$.
We denote by $F_k'$ the handle between $E_k'$ and $E_{k+1}'$, as in Figure~\ref{fig:campo-EF}.

The vector field $\bb$ is extended to be $0$ outside $\cup_{k=0}^\infty (E_k \cup F_k)$. It is extended inside every $F_k$ by choosing a smooth extension in a neighborhood of each handle, whose trajectories are the ones described by the handle. The modulus of $\bb$ is chosen to be between $4^k$ and $4^{k+1}$ in $F_k'$ (notice that $|\bb(x)| = 4^k$ on the top of $E_{k}'$ thanks to \eqref{defn:b-cilindri}).

With this choice, every trajectory in $F_k'$ is not longer than $1$ and the vector field $\bb$ is of size $4^k$. We deduce that the handle is covered in time less than $4^{-k}$.

\begin{figure}[]
\hspace{-9em} \includegraphics[scale=0.62]{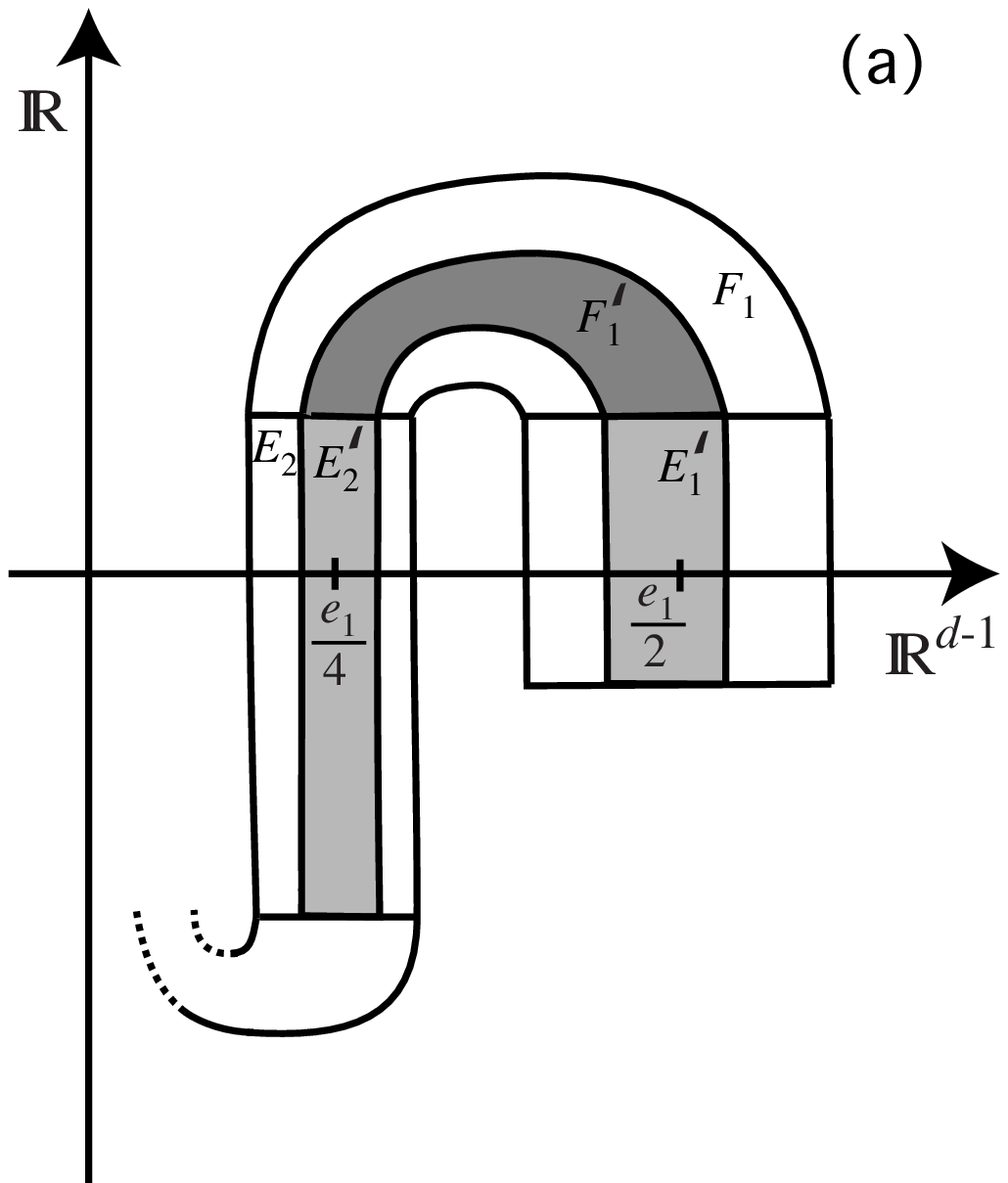}
\hspace{1.5em} \includegraphics[scale=0.62]{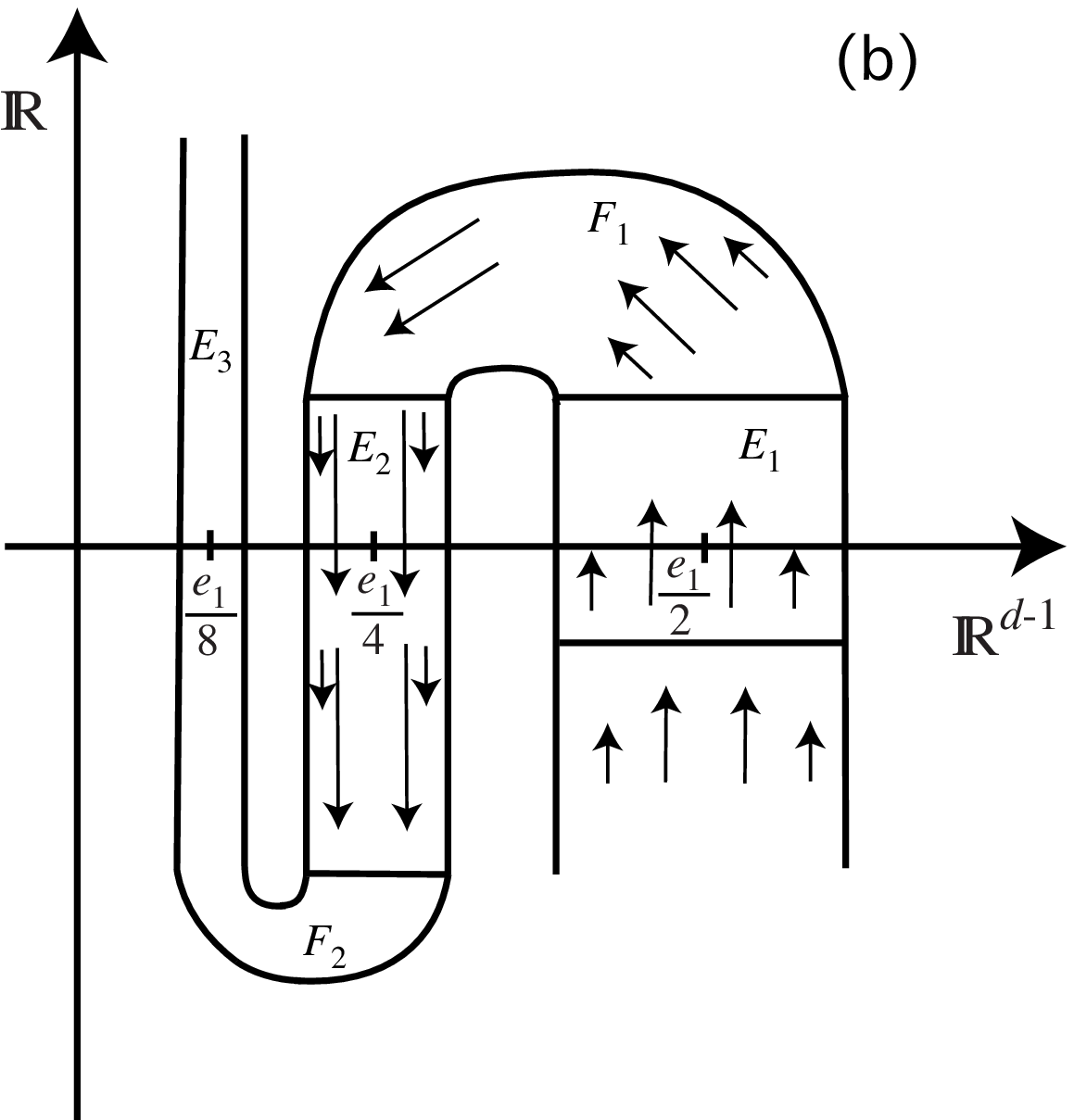}
\caption{The sets $E_k$, $F_k$, $E_k'$, and $F_k'$ and the vector field $\bb$.}
\label{fig:campo-EF}
\end{figure} 

By the construction it is clear that $\bb$ is smooth in $\R^d\setminus \R \ee_d$. We show that $\bb\in W^{1,p}_{\rm loc} (\R^d;\R^d)$ for some $p>1$ by estimating the $W^{1,p}$ norm of $\bb$ in every ball $B_R$. With this estimate, one can easily see that $\bb$ is the limit of smooth vector fields with bounded $W^{1,p}$
norms on $B_R$; it is enough to consider the restriction of $\bb$ to the first $n$ sets $E_k \cup F_k$.

Fix $R>0$. The $W^{1,p}$ norm of $\bb$ in $B_R$ is estimated by 
\begin{equation}
\label{eqn:norma-p-b}
\| \bb \|_{W^{1,p}(B_R)} \leq
\| \bb \|_{W^{1,p}(E_0 \cap B_R)} + \sum_{k=1}^{\infty} \| \bb \|_{W^{1,p}(F_k \cap B_R)} + 
\sum_{k=1}^\infty \| \bb \|_{W^{1,p}(E_k)}.
\end{equation}
The first term is obviously finite (depending on $R$); 
since $B_R$ intersects at most finitely many $F_k $, the second sum in the right-hand side of 
\eqref{eqn:norma-p-b} has only finitely many nonzero terms. As regards the third sum, we compute the 
$W^{1,p}$ norm of $\bb$ in each set $E_k$. For every $k \in \N$
$$\| \bb \|_{L^{p}(E_k)}
\leq
4^k (2R)^{1/p} \Big\| \varphi\Big( \frac{x'-2^{-k}\ee_1}{a_k} \Big)\Big \|_{L^p\big(B^{(d-1)}_{a_k} ( 2^{-k}\ee_1)\big) } 
= 
4^{k} (2Ra_k^{d-1})^{1/p} \|\varphi\|_{L^p(B^{(d-1)}_1)}
$$
and similarly
$$ \| \nabla \bb \|_{L^{p}(E_k)}
\leq
\frac{4^k (2R)^{1/p}}{a_k} \Big\|\nabla \varphi\Big( \frac{x'-2^{-k}\ee_1}{a_k} \Big)\Big \|_{L^p\big(B^{(d-1)}_{a_k} ( 2^{-k}\ee_1)\big) } 
= \frac{4^{k} (2Ra_k^{d-1})^{1/p}}{a_k} \|\nabla \varphi\|_{L^p(B^{(d-1)}_1)}. 
$$
Since $a_k\leq 1$, the series in the right-hand side of \eqref{eqn:norma-p-b} is estimated by
$$\sum_{k=1}^\infty \| \bb \|_{W^{1,p}(E_k)} \leq C(R, \varphi) \sum_{k=1}^\infty
4^{k} a_k^{(d-1)/p-1}
$$
and it is convergent for every $p <d-1$ provided that we take $a_k \leq 8^{-pk/(d-1-p)}$. Hence $\bb \in W^{1,p}(B_R;\R^d)$ for every $R>0$.

To check that $\div \bb \in L^\infty_{\rm loc} (\R^d)$, we notice that $\bb$ is divergence free in $\R^d \setminus\cup_{k=0}^\infty F_k$ and that for every $R>0$ the ball $B_R$ intersects only finitely many handles $F_{k}$; in particular $\bb$ is divergence free in $B_1$. Since $\bb$ is smooth in a neighbourhood of each handle, we deduce that $\div \bb$ is bounded in every $B_R$.

Finally we set $\Sigma=B_{a_1/2} (\ee_1/2) \times [0,1]$ and we show that for every $x\in\Sigma$ the smooth trajectory of $\bb$ starting from $x$ satisfies \eqref{eq:blowup-c}.
The trajectory of $x$ lies by construction in $\cup_{k=0}^\infty (E_k' \cup F_k')$. For every $k\in\N$, the time requested to cross the set $E_k'$ is $2^k/4^k$ and, as observed before, the time requested to cross $F_k'$ is less than $4^{-k}$. Hence 
$$T_{\sXX}(x) \leq \sum_{k=1}^\infty \frac{2^{k}+1}{4^k} \leq 2 \qquad \forall \,x\in E.$$
The other properties in \eqref{eq:blowup-c} are satisfied by construction.
\end{Proof}

In dimension $d=2$, thanks to the smoothness of the vector field built in the previous example outside the $x_2$-axis, there exists only an integral curve of $\bb$ for every $x\in \R^2 \setminus \{x_1=0\}$. Hence, thanks to the superposition principle the previous example satisfies the assumption (\bbb-$\O$) on $\bb$ and therefore provides a two-dimensional counterexample to the proper blow-up of trajectories.
On the other hand, the vector field built in the previous example is not in $BV_{\rm loc}(\R^2;\R^2)$.
We show indeed in the next proposition that for any autonomous $BV_{\rm loc}$ vector field in dimension $d=2$ the behavior of the previous example 
(see Figure~\ref{fig:traiett-campo}) cannot happen and the trajectories must blow up properly. It looks likely that, with $d=2$ and a nonautonomous vector field, one can build an example following the lines of the example in Proposition~\ref{prop:no-prop-blowup}.

\begin{proposition}\label{prop:BV-prop-blowup}
Let  $\bb \in BV_{\rm loc} (\R^2;\R^2)$, $\div \bb \in L^\infty_{\rm loc} (\R^2)$. Then
\begin{equation}\label{eq:blowup-impr}
\liminf_{t\uparrow T_{\sXX}(x)}|\XX(t,x)|= \infty \qquad\text{for $\Leb{2}$-a.e. $x\in\R^2$ such that }\limsup_{t\uparrow T_{\sXX}(x)}|\XX(t,x)|=\infty.
\end{equation} 
\end{proposition}

\begin{figure}[]
\centering 
\includegraphics[scale=0.50]{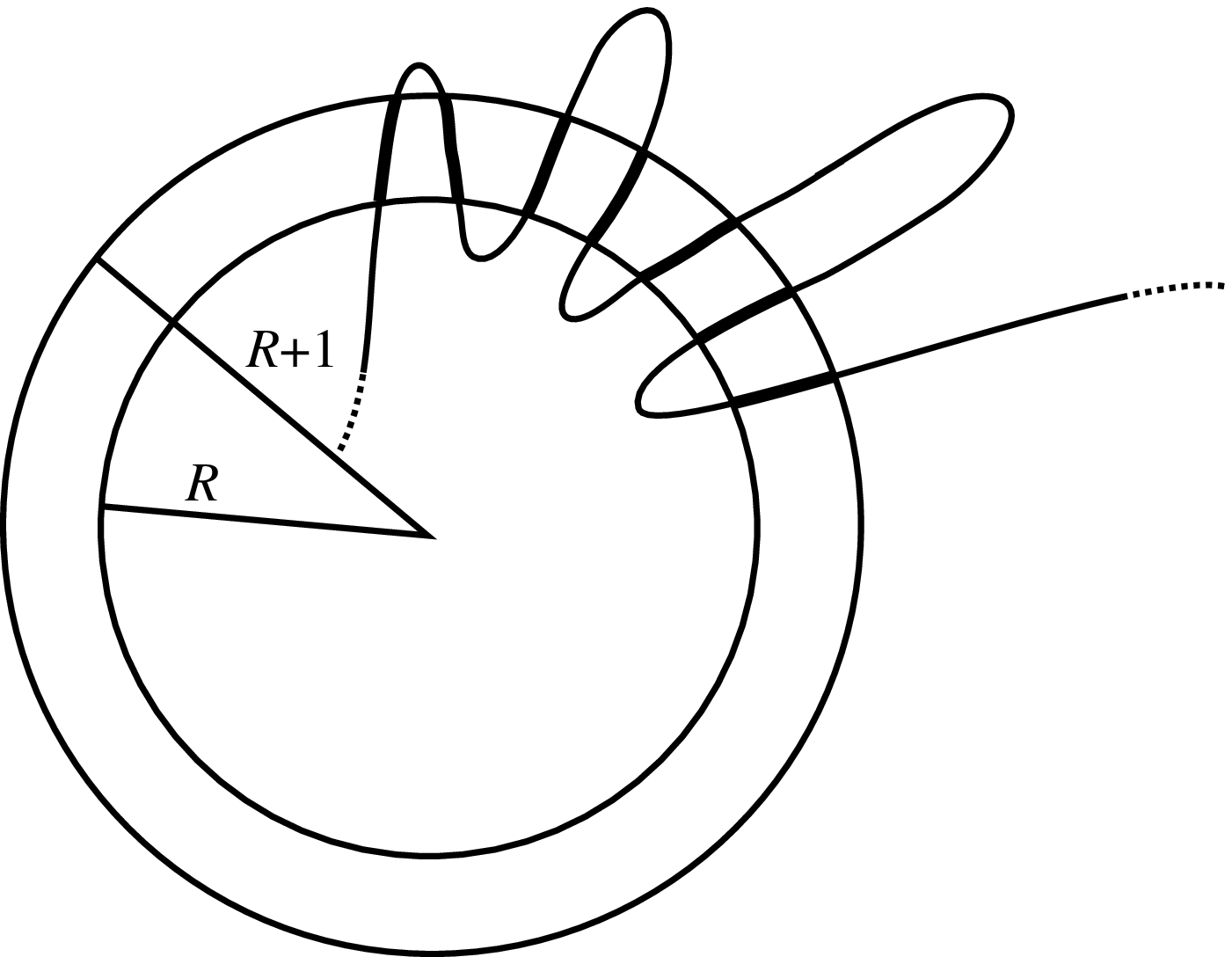}
\caption{For an autonomous vector field $\bb$ in the plane, we consider an integral curve of a suitable representative of $\bb$, namely a vector field which coincides $\Leb{2}$-a.e. with $\bb$. Given $R>0$, the time needed for the integral curve to cross the annulus $B_{R+1} \setminus B_R$ is greater or equal than the constant $\| \esssup_{\partial B_r} |\bb| \|_{L^1(R, R+1)}^{-1}$ (see \eqref{eqn:crossing-time} below).
For this reason, every trajectory can cross only finitely many times the annulus in finite 
time and therefore every unbounded trajectory must blow up properly, as in \eqref{eq:blowup-c}.}
\label{fig:dim2}
\end{figure} 

\begin{Proof}
\noindent {\bf Step 1.} Let $R>0$. We prove that for every vector field $\bb \in BV_{\rm loc} (\R^2;\R^2)$ 
\begin{equation}
\label{eqn:tv-control}
\int_{R}^{R+1} \esssup_{x\in \partial B_r}| \bb(x) | \, dr
\leq 
\frac{1}{2 \pi R}
 \int_{B_{R+1}\setminus B_R} |\bb(x)| \, dx
 +
 |D\bb|(B_{R+1}\setminus B_R).
\end{equation}
For this, let $\bb_\eps$ be a sequence of smooth vector fields which approximate $\bb$ in $BV(B_{R+1} \setminus B_R)$, namely
$$
\lim_{\eps \to 0} |\bb_\eps-\bb|=0 \qquad \mbox{in } L^1(B_{R+1} \setminus B_R), \qquad 
\lim_{\eps \to 0}\int_{B_{R+1} \setminus B_R} |\nabla \bb_\eps(x)| \, dx
=  |D\bb|(B_{R+1}\setminus B_R).
$$
Up to a subsequence (not relabeled) we deduce that for a.e. $r\in (R,R+1)$
$$ \lim_{\eps \to 0} \bb_\eps = \bb \qquad \mbox{in } L^1(\partial B_r;\R^2).$$
Since we can control the supremum of the one dimensional restriction of $\bb_\eps$ to $\partial B_r$ through the $L^1$ norm of $\bb_\eps$ 
and the total variation we have that
$$ \sup_{x\in \partial B_r}| \bb_\eps(x) |
\leq 
\frac{1}{2 \pi r}
 \int_{\partial B_r} |\bb_\eps(x)| \, dx
 +
\int_{\partial B_r} |\nabla \bb_\eps(x)| \, dx.
$$
Hence, integrating with respect to $r$ in $(R,R+1)$, \eqref{eqn:tv-control} holds for $\bb_\eps$: 
$$\int_{R}^{R+1} \sup_{x\in \partial B_r}| \bb_\eps(x) | \, dr
\leq 
\frac{1}{2 \pi R}
 \int_{B_{R+1}\setminus B_R} |\bb_\eps(x)| \, dx
 +
\int_{B_{R+1} \setminus B_R} |\nabla \bb_\eps(x)| \, dx.$$
 Taking the $\liminf$ in both sides as $\eps$ goes to $0$, by Fatou lemma we deduce that
 \begin{equation*}
 \begin{split}
 \int_{R}^{R+1} \esssup_{x\in \partial B_r}| \bb(x) | \, dr
&\leq 
\int_{B_{R+1}\setminus B_R} \liminf_{\eps \to 0} \sup_{x\in \partial B_r}| \bb_\eps(x) | \, dr
\\
&\leq
 \liminf_{\eps \to 0}
\int_{R}^{R+1} \sup_{x\in \partial B_r}| \bb_\eps(x) | \, dr
\\
&\leq
 \lim_{\eps \to 0} \Big(
\frac{1}{2 \pi R}
 \int_{B_{R+1} \setminus B_R}  |\bb_\eps(x)| \, dx
 +
\int_{B_{R+1} \setminus B_R} |D\bb_\eps(x)| \, dx
 \Big)
\\
&=
  \frac{1}{2 \pi R}
\int_{B_{R+1} \setminus B_R}  |\bb(x)| \, dx
 +
 |D\bb|(B_{R+1}\setminus B_R).
 \end{split}
 \end{equation*}

\smallskip
\noindent {\bf Step 2.}
Let $R>0$ and let $\cc: \R^2 \to \R^2$ be a Borel vector field such that
$$f(r) := \sup_{x\in \partial B_r} |\cc(x)| \in L^1(R,R+1).$$
Let $\gamma: [0,\tau] \to \overline{B}_{R+1}\setminus B_R$ be an absolutely continuous integral curve of $\cc$ (namely $\dot \gamma = \cc(\gamma)$ 
$\Leb{1}$-a.e. in $(0,\tau)$) such that $\gamma(0) \in \partial B_R$ and $\gamma(\tau) \in \partial B_{R+1}$. We claim that
\begin{equation}
\label{eqn:crossing-time}
\tau \geq \Big( \int_R^{R+1} f(r) \, dr \Big)^{-1}.
\end{equation}
To prove this, 
we define the nondecreasing function $\sigma: [0,\tau ] \to \R$
\begin{equation}
\label{defn:sigma} \sigma (t) = \max_{s\in [0,t]} |\gamma(s)| \qquad \forall \,t \in [0,\tau];
\end{equation}
we have that $\sigma (0)= R$ and $\sigma (\tau) = R+1$.
For every $s,\,t \in [0, \tau]$ with $s<t$ there holds
$$0\leq \sigma(t)-\sigma(s) \leq \sup_{r\in (s,t]}(|\gamma(r)|-|\gamma(s)|)^+
\leq\int_s^t \bigg|\frac{d}{dr} |\gamma(r)| \bigg|\, dr \leq \int_s^t |\dot \gamma(r)| \, dr.
$$
Thus $\sigma$ is absolutely continuous and $\dot \sigma\leq |\dot\gamma|$ $\Leb{1}$-a.e in $(0,\tau)$.
In addition, for every $t\in (0,\tau)$ such that $\sigma(t) \neq |\gamma(t)|$ the function $\sigma$ is constant in a neighborhood of $t$, hence
$\dot \sigma\leq\chi_{\{ \sigma = |\gamma|\}}|\dot\gamma|$ $\Leb{1}$-a.e. in $(0,\tau)$.
Therefore
$$\dot \sigma(t) \leq 1_{\{ \sigma = |\gamma|\}}(t) |\dot \gamma(t)| 
= 1_{\{ \sigma = |\gamma|\}}(t) |\cc(\gamma(t))|
\leq f(\sigma(t))\qquad\text{for $\Leb{1}$-a.e. $t\in (0,\tau)$.}
$$
By H\"older inequality and the change of variable formula we deduce that
\begin{equation*}
\begin{split}
1 &\leq [\sigma(\tau)-\sigma(0)]^2 \leq \Big( \int_0^\tau \dot\sigma(t) \, dt \Big)^2
 \leq \tau \int_0^\tau [\dot \sigma(t)]^2 \, dt
\\
&\leq
\tau \int_0^\tau \dot \sigma(t)f(\sigma(t)) \, dt
=
\tau \int_R^{R+1} f(\sigma) \, d\sigma,
\end{split}
\end{equation*}
which proves \eqref{eqn:crossing-time}.

\smallskip
\noindent {\bf Step 3.} We conclude the proof. Using the invariance of the concept of maximal regular flow (see Remark~\ref{rem:invaria_flow})
we can work with a well-chosen representative which allows us to apply the estimate in Step 2. For this specific representation of $\bb$, we show that {\it every} integral unbounded trajectory blows up properly.  

For $\Leb{d}$-a.e. $r>0$ the restriction $\bb_r(x) = \bb( rx)$, $x\in {\mathbb S}^1$, of the vector field $\bb$ to $ \partial B_r$ is $BV$.
We remind that every 1-dimensional $BV$ function has a precise representative given at every point by the average of the 
right approximate limit and of the left approximate limit, which exist everywhere. We define the Borel vector field $\cc:\R^2 \to \R$ as
$$\cc(rx) = \text{the precise representative of $\bb_r$ at $x$} \qquad \forall \,x\in {\mathbb S}^1$$
for all $r$ such that $\bb_r\in BV({\mathbb S}^1)$, and $0$ otherwise.
Notice that, by Fubini theorem, $\cc$ coincides $\Leb{2}$-a.e. with $\bb$, and that
$\sup|\cc(r\cdot)|\leq\esssup |\bb(r\cdot)|$ for all $r>0$.

Let us assume by contradiction the existence of $\bar x\in \R^d$ such that $\XX(\cdot, \bar x)$ is an integral curve of the precise representative $\cc$ and
\begin{equation}\label{eq:blowup-contr}
\liminf_{t\uparrow T_{\sXX}(\bar x)}|\XX(t,\bar x)|< \infty, \qquad \limsup_{t\uparrow T_{\sXX}(\bar x)}|\XX(t,\bar x)|=\infty.
\end{equation}
We fix $R>0$ greater than the $\liminf$ in \eqref{eq:blowup-contr}, as in Figure~\ref{fig:dim2} and we define $f(r) := \sup_{x\in \partial B_r} |\cc(x)|$, $r\in [R,R+1]$. Thanks to \eqref{eqn:tv-control} applied to $\cc$, we deduce that $f \in L^1(R,R+1)$. 
Therefore we can apply Step 2 to deduce that every transition from inside $B_R$ to outside $B_{R+1}$ requires at least time 
$1/\|f\|_{L^1(R,R+1)}>0$.
Hence the trajectory $\XX(\cdot , \bar x)$ can cross the set $B_{R+1} \setminus B_R$ only finitely many times in finite time, a contradiction.
\end{Proof}

\subsection{No blow-up criteria}

If one is interested in estimating the blow-up time $T_{\sox}$ of the maximal regular flow, or even if one wants to rule out the
blow up, one may easily adapt to this framework the classical criterion based on the existence of a Lyapunov function $\Psi:\R^d\to [0,\infty]$ satisfying
$\Psi(z)\to \infty$ as $|z|\to\infty$ and
$$ 
\frac{d}{dt}\Psi(x(t))\leq C_\Psi\bigl(1+\Psi(x(t))\bigr)
$$
along absolutely continuous solutions to $\dot x=\bb(t,x)$.
On the other hand, in some cases, by a suitable approximation argument one can exhibit a solution $\mu_t=\rho_t\Leb{d}$ to the continuity equation with
velocity field $\bb$ with $|\bb_t|\rho_t$ integrable. As in \cite[Proposition~8.1.8]{amgisa} (where locally Lipschitz vector fields were considered)
we can use the existence of this solution to rule out the blow-up.

In the next theorem we provide a sufficient condition for the continuity of $\XX$ at the blow-up time, using a global
version of (a-$\O$) and the global bounded compression condition \eqref{eq:pasquetta1}, implied by the global
bound on divergence \eqref{eqn:incompb-o-g}.

\begin{theorem} \label{thm:pasquetta} Let $\bb\in L^1((0,T)\times\O;\R^d)$ satisfy (\bbb-$\O$)
and assume that the maximal regular flow $\XX$ satisfies \eqref{eq:pasquetta1}. Then $\XX(\cdot,x)$ is absolutely continuous in
$[0,T_{\sox}(x)]$ for $\Leb{d}$-a.e. $x\in\O$, and the limit of $\XX(t,x)$ as $t\uparrow T_{\sox}(x)$
belongs to $\partial\Omega$ whenever $T_{\sox}(x)<T$.
\end{theorem}
\begin{Proof} By \eqref{eq:pasquetta1} we have that
\begin{eqnarray*}
  \int_{\O}\int_0^{T_{\O,\sXX}(x)}|\dot\XX(t,x)|\,dx\,dt
 &=&
   \int_{\O}\int_0^{T_{\O,\sXX}(x)}|\bb(t,\XX(t,x))|\,dx\,dt
\\
 &=&
 \int_0^T\int_{ \{T_{\O,\sXX}>t\} }|\bb(t,\XX(t,x))|\,dx\,dt\\
 &\leq&
C_* \int_0^T\int_{\O}|\bb(t,z)|\,dz\,dt.
\end{eqnarray*}
Hence $\XX$ satisfies \eqref{eq:pasquetta1}. Then $\XX(\cdot,x)$ is absolutely continuous in
$[0,T_{\sox}(x)]$ for $\Leb{d}$-a.e. $x\in\O$.
Since the $\limsup V_\O(\XX(t,x))$ as $t\uparrow T_{\sox}$ is $\infty$ whenever $T_{\sox}(x)<T$,
we obtain that in this case the limit of $\XX(t,x)$ as $t\to T_{\sox}(x)$ belongs to
$\partial\O$.
\end{Proof}

In the case $\O=\R^d$ we now prove in a simple criterion for global existence, which allows us 
to recover the classical result in the DiPerna-Lions theory on the existence of a global flow under the growth condition
\begin{equation}
\label{eqn:no-blow-up-cond-DPL-A}
\frac{|\bb(t,x)|}{1+|x|} \in L^1((0,T); L^1(\R^d))+ L^1((0,T); L^\infty(\R^d)).
\end{equation}
As in the previous section, we will use in the next theorem the simplified notation $T_{\sXX}$ for $T_{\R^d,\sXX}$.

\begin{theorem}[No blow-up criterion]\label{thm:noblowup}
Let  $\bb: (0,T) \times \R^d \to \R^d$ be a Borel vector field which satisfies (a-$\O$) and (\bbb-$\O$), and assume that the maximal
regular flow $\XX$ satisfies \eqref{eq:pasquetta1}.
Assume that $\rho_t\in L^\infty\bigl((0,T);L^\infty_+(\R^d)\bigr)$ is a weakly$^*$ continuous solution of the continuity equation 
satisfying the integrability condition
\begin{equation}
\label{no-blowup}
\int_0^T\int_{\R^d}\frac{|\bb(t,x)|}{1+|x|}\rho_t(x)\, dx\,dt<\infty.
\end{equation}
Then $T_{\sXX}(x)=T$ and $\XX(\cdot,x)\in AC([0,T];\R^d)$ for $\rho_0\Leb{d}$-a.e. $x\in\R^d$.
In addition, if the growth condition \eqref{eqn:no-blow-up-cond-DPL-A} holds, then $\rho_t$ satisfying
\eqref{no-blowup} exist for any $\rho_0\in L^1\cap L^\infty(\R^d)$ nonnegative, so that
$\XX$ is defined in the whole $[0,T] \times \R^d$.
\end{theorem}
\begin{Proof} 
For the first part of the statement we apply Theorem \ref{thm:superpo}
to deduce that $\rho_t$ is the marginal at time $t$ of a measure $\eeta\in\Measuresp{C([0,T];\R^d)}$ concentrated on absolutely continuous curves
$\eta$ in $[0,T]$ solving the ODE $\dot\eta=\bb(t,\eta)$.
We then apply Theorem \ref{thm:nosplitting}
to obtain that the conditional probability measures $\eeta_x$ induced by the map $e_0$ are Dirac masses for
$(e_0)_\#\eeta$-a.e. $x$, hence (by uniqueness of the maximal regular flow)
$\rho_t$ is transported by $\XX$. Notice that, as a consequence of the fact that $\eeta$ is concentrated on absolutely continuous curves
in $[0,T]$, the flow is globally defined on $[0,T]$, thus $T_{\sXX}(x)=T$.


For the second part, under assumption \eqref{eqn:no-blow-up-cond-DPL-A} the existence of 
a nonnegative and weakly$^*$ continuous solution of the continuity equation $\rho_t$ 
in $L^\infty\bigl((0,T);L^1\cap L^\infty(\R^d)\bigr)$ can be achieved by a simple smoothing
argument. So, the bound in $L^1\cap L^\infty$ on $\rho_t$ can be combined with
\eqref{eqn:no-blow-up-cond-DPL-A} to obtain \eqref{no-blowup}.
\end{Proof}

\begin{remark}{\rm
We remark that if only a local bound on the divergence is assumed as in Section~\ref{sec:local}, the growth assumption \eqref{eqn:no-blow-up-cond-DPL-A} is not enough to guarantee that the trajectories of the regular flow do not blow up. 
On the other hand, it can be easily seen that if we assume that $\bb$ satisfies (a-$\R^d$), (b-$\R^d$), \eqref{eqn:incompb-o-L} and $|\bb(t,x)|/(1+|x|) \in L^1((0,T); L^\infty(\R^d))$, every integral curve of $\bb$ cannot blowup in finite time and therefore the maximal regular flow satisfies $T_{\sXX}(x)=T$ and $\XX(\cdot,x)\in AC([0,T];\R^d)$ for a.e. $x\in\R^d$.

Theorem~\ref{thm:noblowup} is useful in applications when one constructs solutions by approximation. For instance, for the Vlasov-Poisson system in dimension $d=2$ and $3$, this result can be used to show that trajectories which transport a bounded solution with finite energy do not explode in the phase space (see~\cite{amcomfig2}).
}\end{remark}

\section*{Appendix: On the local character of the assumption (\bbb-$\O$)}

Here we prove that the property (\bbb-$\O$) of Section~\ref{sec:ass-a-c} is local, in analogy with the other assumptions ((a-$\O$) and the local bounds on
distributional divergence)
made throughout this paper. More precisely, the following assumption is equivalent to (\bbb-$\O$):

\begin{itemize}
\item[(\bbb'-$\O$)] for any $t_0\geq 0, x_0 \in \O$ there exists $\eps := \eps(t_0,x_0)>0$ such that for any nonnegative $\bar\rho\in L^\infty(\R^d)$ with compact support contained in $B_\eps(x_0)\subset \O$ and any closed interval $I=[a,b]\subset [t_0-\eps,t_0+\eps] \cap [0,T]$, the continuity equation
$$
\frac{d}{dt}\rho_t+{\rm div\,}(\bb\rho_t)=0\qquad\text{in $(a,b)\times\R^d$}
$$
has at most one weakly$^*$ continuous solution $I\ni t\mapsto\rho_t\in\mathcal L_{I,\O}$ with $\rho_a=\bar\rho$ and $\rho_t$ compactly supported in $B_\eps(x_0)$ for every $t\in [a,b]$.
\end{itemize}

\begin{lemma}\label{lemma:purely-local}
If the assumptions (a-$\O$) and (\bbb'-$\O$) on the vector field $\bb$ are satisfied, then (\bbb-$\O$) is satisfied.
\end{lemma}
\begin{Proof}
\noindent {\bf Step 1.} 
Let $\eeta\in\Probabilities{C([a,b];\R^d)}$, $0 \leq a < b\leq T$, be concentrated on absolutely continuous curves
$\eta\in AC( [a,b]; K)$ for some $K \subset \O$ compact, solving the ODE $\dot\eta=\bb(t,\eta)$ $\Leb{1}$-a.e. in $(a,b)$, and such that $(e_t)_\#\eeta \leq C \Leb{d}$ for any $t\in [0,T]$. 
We claim that the conditional probability measures $\eeta_x$ induced by the map $e_a$ are Dirac masses for $(e_a)_\#\eeta$-a.e. $x$.

To this end, for $s,t\in [a,b]$, $s<t$, we  denote by $\Sigma^{s,t}:C([a,b]; \R^d)\to C([s,t];\R^d)$ the map induced by restriction to $[s,t]$, namely $\Sigma^{s,t}(\eta)=\eta\vert_{[s,t]}$.
For $(e_a)_\#\eeta$-a.e. $x\in \R^d$ we define $\tau (x)$ the first splitting time of $\eeta_x$, namely the infimum of all $t> a$ such that $(\Sigma^{a,t})_\#\eeta_x$ is not a Dirac mass. We agree that $\tau (x) = T$ if $\eeta_x$ is a Dirac mass.
We also define the splitting point $B(x)$ as $\eta (\tau (x))$ for any $\eta \in \supp \eeta_x$.
By contradiction, we assume that the set $\{x \in \R^d:\tau (x)<T\}$ has positive $(e_a)_\#\eeta$ measure.

For every $t_0>0$ and $x_0 \in \R^d$ let $ \eps(t_0,x_0)>0$ be as in (\bbb'-$\O$).
By a covering argument, we can take a finite cover of $[a,b] \times K$ with sets of the form 
$$I_{t_0,x_0, \eps(t_0,x_0)} = (t_0-\eps(t_0,x_0),t_0+\eps(t_0,x_0)) \times B_{\eps(t_0,x_0)/2} (x_0).$$
We deduce that there exists $t_0>0$ and $x_0 \in \R^d$ such that the set
\begin{equation}
E_0 := \{x \in \R^d:\tau (x)<T, \; (\tau(x), B(x)) \in I_{t_0,x_0, \eps(t_0,x_0)} \}
\end{equation}
has positive $(e_a)_\#\eeta$ measure.

For every $p, q\in \Q$ with $a \leq p< q \leq b$
we define the open set
$$E_{p,q} := \{ \eta \in  C([a,b]; \R^d): \eta([p,q]) \subset B_{\eps(t_0,x_0)/2} (x_0) \}.$$

We claim that there exist a set $E_1 \subset E_0$ and $p,q\in \Q\cap [a,b]$, $p<q$ such that $(e_a)_\#\eeta (E_1)>0$ and
 for every $x\in E_1$ the measure $\Sigma^{p,q}_\#(1_{E_{p,q}} \eeta_x)$ is not a Dirac delta.  

To this end, it is enough to show that for a.e. $x \in E_0$ 
there exist $p_x,q_x\in \Q \cap [a,b]$, $p_x<q_x$ such that
$\Sigma^{p_x,q_x}_\#(1_{E_{p_x,q_x}} \eeta_x)$ is not a Dirac delta.  

Let us consider $\eta_1 \in \supp \eeta_x$; it satisfies $\eta_1(\tau(x)) = B(x) \in B_{\eps(t_0,x_0)/2} (x_0)$. Let $p_x,q_x$ be chosen such that $\eta_1([p_x, q_x]) \subseteq B_{\eps(t_0,x_0)/2} (x_0)$. By definition of $\tau(x)$ we know that $\Sigma^{p_x,q_x}_\# \eeta_x$ is not a Dirac delta. Hence there exists $\eta_2 \in C([a,b]; \R^d)$ such that $\eta_2 \in \supp(\eeta_x)$, $\eta_2(\tau(x) ) = B(x)$, $\eta_1(t)\neq \eta_2(t)$ for every $t\in [a,\tau(x)]$, $\eta_1(t)\neq \eta_2(t)$ for some $t\in [\tau(x),q_x]$. Up to reducing $q_x$, we can assume that $\Sigma^{p_x,q_x}(\eta_1), \Sigma^{p_x,q_x}(\eta_2)$ are curves whose image is contained in $B_{\eps(t_0, x_0)/2}(x_0)$, so that $\eta_1,\eta_2 \in E_{p_x,q_x}$, and which do not coincide. Moreover, since $\supp (\Sigma^{p_x,q_x}_\# \eeta_x ) = \Sigma^{p_x,q_x}(\supp \eeta_x )$,
we deduce that both $\Sigma^{p_x,q_x}(\eta_1)$ and $\Sigma^{p_x,q_x}(\eta_2)$ belong to the support of $\Sigma^{p_x,q_x}_\#( \eeta_x)$
and hence $\Sigma^{p_x,q_x}_\#(1_{E_{p_x,q_x}} \eeta_x) = 1_{\Sigma^{p_x,q_x} (E_{p_x,q_x})} \Sigma^{p_x,q_x}_\#\eeta_x$ is not a Dirac delta.  

%

Let $\delta>0$ be small enough so that $E_\delta=  E_1  \cap \{x: \eeta_x(E_{p,q})\geq \delta\} $ has positive $(e_a)_\#\eeta$-measure.
We introduce the probability measure $\tilde \eeta\in\Probabilities{C([a,b];\R^d)}$
$$\tilde \eeta := ((e_a)_\#\eeta \res E_\delta) \otimes \Big(\frac{1_{E_{p,q}} }{\eeta_x(E_{p,q})}\eeta_x\Big) = ((e_a)_\#\eeta \res E_\delta) \otimes \tilde \eeta_x,$$
which is nonnegative, and less than or equal to $\eeta/\delta$.
Moreover $\Sigma^{p,q}_\# \tilde \eeta \in\Probabilities{C([p,q];\R^d)}$ is concentrated on curves in $B_{\eps(t_0,x_0)/2}(x_0)$, and
$$
\Sigma^{p,q}_\# \tilde \eeta_x = \frac{\Sigma^{p,q}_\# (1_{E_{p,q}} \eeta_x)}{\eeta_x(E_{p,q})} \text{ is not a Dirac mass for $(e_a)_\#\eeta$-a.e. $x\in E_\delta$.}
$$
Applying Theorem~\ref{thm:nosplitting} with  $\llambda = \Sigma^{p,q}_\#\tilde \eeta$, $\Omega = B_{\eps(t_0,x_0)}(x_0)$, in the time interval $[p,q]$, and thanks to the local uniqueness of bounded, nonnegative solutions of the continuity equation in $I_{t_0,x_0, \eps(t_0,x_0)} $, which in turn follows from (\bbb'-$\O$), we deduce that the disintegration $\Sigma^{p,q}_\#\tilde \eeta_x$ of $\Sigma^{p,q}_\#\tilde \eeta$ induced by $e_a$ is a Dirac mass for $(e_a)_\#\eeta$-a.e. $x\in E_\delta$. By the uniqueness of the disintegration, we obtain a contradiction.

\smallskip
\noindent {\bf Step 2.} Let $\mu^1$ and $\mu^2$ be two solutions of the continuity equation as in (\bbb) with the same initial datum. Let $\eeta^1,\eeta^2\in\Probabilities{C([a,b];\R^d)}$ be the representation of $\mu^1$ and $\mu^2$ obtained through the superposition principle; they are concentrated on absolutely continuous integral curves of $\bb$ and they satisfy $\mu^i_t=(e_t)_\#\eeta^i$ for any $t\in [0,T]$, $i=1,2$. Since there exists a compact set $K \subset \O$ such that $\mu^i_t$ is concentrated on $K$ for every $t\in [0,T]$, $\eeta^i$ is concentrated on absolutely continuous curves contained in $K$ for $i=1,2$. Then by the linearity of the continuity equation  $(e_t)_\# [ (\eeta_1+ \eeta_2)/2 ] = (\mu^1_t+\mu^2_t )/2$ is still a solution to the continuity equation; by Step 1 we obtain that $(\eeta^1_x+ \eeta^2_x) /2$ are Dirac masses for $\mu_0$-a.e. $x$. This shows that $\eeta^1_x= \eeta^2_x$ for $\mu_0$-a.e. $x$ and therefore that $\mu^1_t=\mu^2_t$ for every $t\in [0,T]$.
\end{Proof}

\end{document}